\documentclass[oneside,english]{amsart}
\usepackage[T1]{fontenc}
\usepackage[latin9]{inputenc}
\usepackage{babel}
\usepackage{amstext}
\usepackage{amsthm}
\usepackage{amssymb}
\usepackage[unicode=true,pdfusetitle,
 bookmarks=true,bookmarksnumbered=false,bookmarksopen=false,
 breaklinks=false,pdfborder={0 0 1},backref=page,colorlinks=false]
 {hyperref}

\makeatletter
\numberwithin{equation}{section}
\numberwithin{figure}{section}
\theoremstyle{plain}
\newtheorem{thm}{\protect\theoremname}[section]
  \theoremstyle{definition}
  \newtheorem{defn}[thm]{\protect\definitionname}
  \theoremstyle{remark}
  \newtheorem{rem}[thm]{\protect\remarkname}
  \theoremstyle{plain}
  \newtheorem{lem}[thm]{\protect\lemmaname}
  \theoremstyle{plain}
  \newtheorem{prop}[thm]{\protect\propositionname}
  \theoremstyle{remark}
  \newtheorem{notation}[thm]{\protect\notationname}
  \theoremstyle{plain}
  \newtheorem{question}[thm]{\protect\questionname}
  \theoremstyle{remark}
  \newtheorem*{claim*}{\protect\claimname}
  \theoremstyle{plain}
  \newtheorem{cor}[thm]{\protect\corollaryname}
  \theoremstyle{definition}
  \newtheorem{example}[thm]{\protect\examplename}

\makeatletter
\newcommand{\dotminus}{\mathbin{\text{\@dotminus}}}
\newcommand{\@dotminus}{%
  \ooalign{\hidewidth\raise1ex\hbox{.}\hidewidth\cr$\m@th-$\cr}%
}
\makeatother


\makeatother

  \providecommand{\claimname}{Claim}
  \providecommand{\corollaryname}{Corollary}
  \providecommand{\definitionname}{Definition}
  \providecommand{\examplename}{Example}
  \providecommand{\lemmaname}{Lemma}
  \providecommand{\notationname}{Notation}
  \providecommand{\propositionname}{Proposition}
  \providecommand{\questionname}{Question}
  \providecommand{\remarkname}{Remark}
\providecommand{\theoremname}{Theorem}

\begin{document}
\global\long\def\SR#1#2{\mathsf{SR}_{#1}^{#2}}
\global\long\def\MLR#1#2{\mathsf{MLR}_{#1}^{#2}}
\global\long\def\KR#1#2{\mathsf{KR}_{#1}^{#2}}
\global\long\def\id{\operatorname{id}}
\global\long\def\eval{\operatorname{eval}}
\global\long\def\dom{\operatorname{dom}}
\global\long\def\supp{\operatorname{supp}}
\global\long\def\probmeas{\mathcal{M}_{1}}
\global\long\def\finitemeas{\mathcal{M}}
\global\long\def\radonmeas{\mathcal{M}_{\textnormal{loc}}}
\global\long\def\meas{\mathcal{M}}
\global\long\def\locflag{[\textnormal{loc}]}
\global\long\def\cont{\mathcal{C}}

\title{Schnorr randomness for noncomputable measures}

\author{Jason Rute}

\address{Department of Mathematics\\
Pennsylvania State University\\
University Park, PA 16802}

\email{jmr71@math.psu.edu}

\urladdr{http://www.personal.psu.edu/jmr71/}


\thanks{We would like to thank Mathieu Hoyrup for the results concerning
uniform Schnorr sequential tests in Subsection~\ref{subsec:uniform-sequential-tests}
and Christopher Porter for pointing out the Schnorr-Fuchs definition
mentioned in Subsection~\ref{subsec:Schnorr-Fuchs-def}. We would
also like to thank Jeremy Avigad, Peter Gács, Mathieu Hoyrup, Bjørn
Kjos-Hanssen, and two anonymous referees for corrections and comments.\\
\indent Started on January 16, 2015. Last updated on \today.}
\begin{abstract}
This paper explores a novel definition of Schnorr randomness for noncomputable
measures. We say $x$ is uniformly Schnorr $\mu$-random if $t(\mu,x)<\infty$
for all lower semicomputable functions $t(\mu,x)$ such that $\mu\mapsto\int t(\mu,x)\,d\mu(x)$
is computable. We prove a number of theorems demonstrating that this
is the correct definition which enjoys many of the same properties
as Martin-Löf randomness for noncomputable measures. Nonetheless,
a number of our proofs significantly differ from the Martin-Löf case,
requiring new ideas from computable analysis.
\end{abstract}

\maketitle

\section{Introduction}

Algorithmic randomness is a branch of mathematics which gives a rigorous
formulation of randomness using computability theory. The first algorithmic
randomness notion, Martin-Löf randomness, was formulated by Martin-Löf
\cite{Martin-Lof:1966mz} and has remained the dominant notion in
the literature. Schnorr \cite{Schnorr:1969ly}, however, felt that
Martin-Löf randomness was too strong, and introduced a weaker, more
constructive, randomness notion now known as Schnorr randomness. Both
Martin-Löf and Schnorr randomness play an important role in computable
analysis and computable probability theory. For example, the Martin-Löf
randoms are exactly the points of differentiability for all computable
functions $f\colon[0,1]\rightarrow\mathbb{R}$ of bounded variation
\cite{Brattka.Miller.Nies:}. Similarly, the Schnorr randoms are exactly
the Lebesgue points for all functions $f\colon[0,1]\rightarrow\mathbb{R}$
computable in the $L^{1}$-norm \cite{Rute:2013pd,Pathak.Rojas.Simpson:2014}. 

Algorithmic randomness is formulated through the idea of ``computable
tests.'' Specifically, if $\mu$ is a computable measure on a computable
metric space $\mathbb{X}$, then (in this paper) a test for Martin-Löf
$\mu$-randomness is a lower semicomputable function $t\colon\mathbb{X}\rightarrow[0,\infty)$
such that $\int_{{}}t\,d\mu<\infty$. A point $x$ passes the test
$t$ if $t(x)<\infty$, else it fails the test. A point $x$ is Martin-Löf
$\mu$-random if $x$ passes all such tests $t$. Schnorr randomness
is the same, except that we also require that $\int_{{}}t\,d\mu$
is computable for each test $t$. (We present the full details in
the paper.)

While, historically, algorithmic randomness was mostly studied for
computable probability measures, there were a few early papers investigating
Martin-Löf randomness for arbitrary noncomputable probability measures.
One was by Levin \cite{Levin:1976uq} using the concept of a ``uniform
test''\textemdash that is, a test $t$ which takes as input a pair
$(\mu,x)$ and for which $t(\mu,x)=\infty$ if and only if $x$ is
$\mu$-random. Gács, later, \cite{Gacs:2005} modified Levin's uniform
test approach.\footnote{Levin \cite{Levin:1976uq} required that uniform tests have two additional
properties, called monotonicity and concavity, while Gács \cite{Gacs:2005}
removed these conditions. The two approaches lead to different definitions
of Martin-Löf $\mu$-randomness for noncomputable measures $\mu$.
Gács's approach is now standard.}  Separately, Reimann \cite{Reimann:2008vn} (also Reimann and Slaman
\cite{Reimann.Slaman:2015}) gave an alternate definition using the
concept of a ``relativized test''\textemdash that is a test $t$
which is computable from (a name for) $\mu$. Day and Miller \cite{Day.Miller:2013}
showed the Levin-Gács and Reimann definitions are equivalent. Recently,
there have been a number of papers investigating Martin-Löf randomness
for noncomputable measures, e.g.~\cite{Kjos-Hanssen:2010aa,Bienvenu.Gacs.Hoyrup.ea:2011,Bienvenu.Monin:2012,Day.Miller:2013,Day.Reimann:2014,Reimann.Slaman:2015}.
These results have applications to effective dimension \cite{Reimann:2008vn},
the ergodic decomposition for computable measure-preserving transformations
\cite{Hoyrup:2011}, and the members of random closed sets \cite{Kjos-Hanssen:2009,Diamondstone.Kjos-Hanssen:2012}\textemdash just
to name a few.

In stark contrast, Schnorr randomness for noncomputable measures has
remained virtually untouched. The first goal of this paper is to give
a proper definition of Schnorr randomness for noncomputable measures.
Our definition is based on the Levin-Gács uniform tests. 

The second goal of this paper is to convince the reader that our definition
is the correct one. We will do this by showing that the major theorems
concerning Martin-Löf randomness for noncomputable measures also hold
for (our definition of) Schnorr randomness for noncomputable measures.
While many of the theorems in this paper are known to hold for Martin-Löf
randomness, the Schnorr randomness versions require different arguments,
using new ideas and tools from computable analysis. However, our proofs
naturally extend to Martin-Löf randomness as well. In some cases,
we even prove new results about Martin-Löf randomness.

\subsection{Uniform verse nonuniform reasoning}

There are a number of reasons that Schnorr randomness has remained
less dominant up to this point. The first is historical: Martin-Löf
randomness came first. (Also, much of Schnorr's work, particularly
his book \cite{Schnorr:1971rw}, was written in German and never translated
into English.)

However, there is also another reason: Many consider Schnorr randomness
to be less well behaved than Martin-Löf randomness \cite[\S7.1.2]{Downey.Hirschfeldt:2010}.
Generally two results are given in support of this claim:
\begin{enumerate}
\item Schnorr randomness does not have a universal test.
\item Van Lambalgen's Theorem fails for Schnorr randomness.
\end{enumerate}
As for the first point, Martin-Löf showed that there is one universal
test $t$ for Martin-Löf randomness such that $x$ is Martin-Löf random
iff $x$ passes $t$. In contrast, for every Schnorr test $t$ there
is a computable point (hence not Schnorr random) which fails $t$.
This latter result, while an inconvenience in proofs, actually shows
that Schnorr randomness is more constructive. If an a.e.\ theorem
holds for Schnorr randomness (for example, the strong law of large
numbers), then we can generally construct a computable pseudo-random
object satisfying this a.e.\ theorem. 

As for the second point, Van Lambalgen's Theorem says that a pair
$(x,y)$ is Martin-Löf random if and only if $x$ is Martin-Löf random
and $y$ is Martin-Löf random relative to $x$. Whether Van Lambalgen's
Theorem holds for Schnorr randomness depends on how one interprets
``$y$ is random relative to $x$.'' If we use a uniform test approach
(similar to the Levin-Gács uniform tests) then it holds \cite{Miyabe.Rute:2013}.
If we use a non-uniform relativized test approach (similar to Reimann's
relativized tests) then it does not hold \cite{Yu:2007}. Uniform
approaches were more common in the earlier work of Martin-Löf, Levin,
Schnorr, and (to a lesser degree) Van Lambalgen.\footnote{Indeed, to the extent that Martin-Löf \cite[\S IV]{Martin-Lof:1966mz},
Levin \cite{Levin:1976uq}, and Schnorr \cite[\S 24]{Schnorr:1971rw}
explored randomness for noncomputable measures\textemdash usually
Bernoulli measures\textemdash their approaches were uniform. One of
the two (equivalent) approaches to relative randomness in Van Lambalgen
\cite[Def.~5.6]{Lambalgen:1990} is also uniform.}  However, now-a-days it is more common to see nonuniform relativized
test approaches. (To be fair, the distinction between uniform and
nonuniform reasoning in randomness\textemdash and computability theory
in general\textemdash is quite blurred. This is further exacerbated
by the fact that for Martin-Löf randomness, the uniform and nonuniform
approaches are equivalent. Nonetheless, one area in computability
theory where the distinction between uniform and nonuniform approaches
are different is the truth-table degrees and the Turing degrees. Indeed
the uniform approach to Schnorr randomness was originally called \emph{truth-table Schnorr randomness}
\cite{Franklin.Stephan:2010}.)

This paper is built on the uniform approach, and we believe this goes
far in explaining why Schnorr randomness behaves the way it does.
Nonetheless, we also briefly look at the nonuniform approach in Subsection~\ref{subsec:seq-test-rel-to-name}.

\subsection{Finite measures on computable metric spaces}

In this paper, we take a general approach. Instead of working with
only Cantor space, we explore randomness for all computable metric
spaces. We do this because many of the most interesting applications
of randomness occur in other spaces. For example, Brownian motion
is best described as a probability measure on the space $\cont([0,\infty))$
which is not even a locally compact space. Moreover, the finite-dimensional
vector space $\mathbb{R}^{d}$ is a natural space to do analysis,
and any reasonable approach to randomness should be applicable there.

Not only do we consider other spaces, but we also consider finite
Borel measures which may not necessarily be probability measures.
While, probability theory is mostly concerned with probability measures,
other applications of measure theory rely on more general Borel measures.
In particular, potential theory (which has had some recent connections
with randomness \cite{Reimann:2008vn,Diamondstone.Kjos-Hanssen:2012,Allen.Bienvenu.Slaman:2014,Miller.Rute:})
uses finite Borel measures on $\mathbb{R}^{d}$.

\subsection{Outline}

The paper is organized as follows. Section~\ref{sec:background}
contains background on computable analysis and computable measure
theory. Most of the material can be found elsewhere.

Section~\ref{sec:SR-comp} contains a review of Schnorr and Martin-Löf
randomness for \emph{computable probability} measures, while Section~\ref{sec:SR-noncomp}
introduces our definitions of Schnorr and Martin-Löf randomness for
\emph{noncomputable} measures.

While our uniform test definition given in Section~\ref{sec:SR-noncomp}
is elegant in its simplicity, it is difficult to work with. In Section~\ref{sec:tests-restricted}
we show that we may restrict our uniform test to any effectively closed
set of measures, and then in Section~\ref{sec:basic-results} we
use this fact to prove a number of basic facts about Schnorr randomness
for noncomputable measures.

Randomness for noncomputable measures allows us to state and prove
a number of variations of Van Lambalgen's Theorem. These variations,
while useful facts in their own right, help to justify that our definition
of Schnorr randomness for noncomputable measures is the correct one.
In Section~\ref{sec:SR-product-meas} we state and prove Van Lambalgen's
Theorem for noncomputable product measures. In particular, we characterize
which pairs $(x,y)$ are random for a noncomputable product measure
$\mu\otimes\nu$. In Section~\ref{sec:SR-kernels} we state and prove
Van Lambalgen's Theorem for a more general class of measures on a
product space, namely measures given by a probability measure and
a kernel $\kappa$. In this paper, we only consider continuous kernels,
but with more work we could prove a similar theorem for measurable
kernels.

In Section~\ref{sec:SR-maps}, we turn to measure preserving maps
$T\colon(\mathbb{X},\mu)\rightarrow(\mathbb{Y},\nu)$, proving a number
of results relating the randomness on the space $(\mathbb{X},\mu)$
to the space $(\mathbb{Y},\nu)$. This relates to the randomness preservation
theorems and no-randomness-from-nothing theorems which can be found
in other papers. In particular, we prove the following quite general
result: Given a continuous measure-preserving map $T\colon(\mathbb{X},\mu)\rightarrow(\mathbb{Y},\nu)$,
the following are equivalent:
\begin{enumerate}
\item $x$ is $\mu$-random relative to $T$, and $y=T(x)$.
\item $y$ is $\nu$-random relative to the pair $(T,\mu)$, and $x$ is
$\mu(\cdot\mid T=y)$ random relative to the triple $(T,\mu,y)$.
\end{enumerate}
This theorem\textemdash which is new for both Schnorr and Martin-Löf
randomness, as well as for both computable and noncomputable measures\textemdash lets
us easily pass between randomness for $(\mathbb{X},\mu)$ and randomness
for $(\mathbb{Y},\nu)$.

Despite the theorems proved so far, we still need to show that our
definition is practical when relativizing Schnorr randomness results
to noncomputable measures. There are two obstacles we must overcome.
Many results for Schnorr randomness require reasoning which is not
uniform in the measure $\mu$, but is uniform in a name for the measure.
(A common example is that every probability space can be decomposed
into regions with null boundary, which makes the space look and act
like Cantor space. However, this decomposition is only uniformly computable
in a name for the measure, not uniform computable in the measure itself.)
The second problem is that our definition of Schnorr randomness relies
on integral tests, while a more typical type of test used in Schnorr
randomness is a sequential test (commonly just called a Schnorr test).
We handle these issues in Section~\ref{sec:useful-characterization}.
Firstly, we show in Theorem~\ref{thm:SR-name} that it does not matter
if we use tests which are uniform in the measure, or uniform in the
Cauchy name for the measure. This result is similar to that of Day
and Miller \cite{Day.Miller:2013} for Martin-Löf randomness except
that our proof is different and more general. Their proof neither
works for Schnorr randomness nor noncompact spaces, while our proof
works for both. Secondly, we show how one can use Theorem~\ref{thm:SR-name}
to relativize the usual sequential test proof that Schnorr randomness
is stronger than Kurtz randomness. Thirdly, we give a sequential test
characterization of Schnorr randomness for noncomputable measures.

In Section~\ref{sec:Other-bad-defs}, we address some alternate possible
definitions of Schnorr randomness for noncomputable measures. The
first is a nonuniform approach. We show this has less desirable properties.
The second is a martingale definition of Schnorr and Fuchs. We show
this is a ``blind definition'' which does not use any computability
theoretic properties of the measure. The third is a uniform sequential
test definition of Hoyrup. We give Hoyrup's unpublished result that
all uniform sequential tests (of a certain type) are trivial\textemdash justifying
the integral test approach we take in this paper.

Last, we end with Section~\ref{sec:conclusion} which contains discussion,
open questions, and further directions. This relatively small paper
does not come close to addressing all the topics which have been investigated
regarding Martin-Löf randomness for noncomputable measures. There
is still a lot of work to do and open questions to answer.

\section{\label{sec:background}Computable analysis and computable measure
theory}

As for notation, we denote the space of infinite binary sequences
(Cantor space) as $\{0,1\}^{\mathbb{N}}$ and the finite binary strings
as $\{0,1\}^{*}$. For a string $\sigma\in\{0,1\}^{*}$ the cylinder
set of sequences extending that string is denoted $[\sigma]$. The
same conventions will be used for the space of infinite sequences
of numbers $\mathbb{N}^{\mathbb{N}}$ (Baire space).

\subsection{Computable analysis}

We assume the reader is familiar with computability theory in the
countable spaces $\mathbb{N}$, $\mathbb{Q}$, $\mathbb{N}\times\mathbb{N}$,
and $\{0,1\}^{*}$, as well as the uncountable spaces $\{0,1\}^{\mathbb{N}}$
and $\mathbb{N}^{\mathbb{N}}$\textemdash as can be found in a standard
computability theory text, e.g.~\cite{Soare:1987dp}. Also, we assume
some familiarity with computability theory on $\mathbb{R}$ (e.g.~\cite{Brattka.Hertling.Weihrauch:2008}).
In particular, a function $f\colon\mathbb{N}\rightarrow\mathbb{R}$
is computable if there is a computable map $g\colon\mathbb{N}\times\mathbb{N}\rightarrow\mathbb{Q}$
such that 
\[
\forall n\ \forall j\geq i\ |g(n,j)-g(n,i)|\leq2^{-i}
\]
and such that $f(n)=\lim_{i}g(n,i)$ for all $n$.
\begin{defn}
A \emph{computable metric space} $\mathbb{X}=(X,d,(x_{i})_{i\in\mathbb{N}})$
is a complete separable metric space $(X,d)$ along with a dense sequence
of points $x_{i}\in X$, such that $(i,j)\mapsto d(x_{i},x_{j})$
is computable. We refer to the points $x_{i}$ as \emph{basic points},
the rational open balls $B$ with basic point centers as \emph{basic open balls},
and the rational closed balls $\overline{B}$ with basic point centers
as \emph{basic closed balls}. As for notation, we will write, say,
$x\in\mathbb{X}$ and $A\subseteq\mathbb{X}$ instead of the more
pedantic $x\in X$ and $A\subseteq X$.

A \emph{Cauchy name} for $y\in\mathbb{X}$ is a function $h\in\mathbb{N}^{\mathbb{N}}$
satisfying
\[
\forall j\geq i\ d(x_{h(j)},x_{h(i)})\leq2^{-i}
\]
such that $y=\lim_{i}x_{h(i)}$. A point $y\in\mathbb{X}$ is \emph{computable}
if it has a computable Cauchy name.
\end{defn}
The spaces $\mathbb{N}$, $\{0,1\}^{\mathbb{N}}$, $\mathbb{N}^{\mathbb{N}}$,
$[0,\infty)$, and $\mathbb{R}$ are all computable metric spaces.
If $\mathbb{X}$ and $\mathbb{Y}$ are computable metric spaces then
so are $\mathbb{X}\times\mathbb{Y}$, $\mathbb{X}^{\mathbb{N}}$,
etc.
\begin{rem}
Every computable metric space is a Polish space (a complete separable
metric space). As in Polish space theory, we are concerned with the
metric on the space $\mathbb{X}$, only in so far as it generates
a certain computable topological structure on $\mathbb{X}$. Two computable
metric spaces $\mathbb{X}$ and $\mathbb{Y}$ are \emph{equivalent}
in this sense if there is a computable homeomorphism between the spaces\textemdash that
is computable maps (see below) $f\colon\mathbb{X}\rightarrow\mathbb{Y}$
and $g\colon\mathbb{Y}\rightarrow\mathbb{X}$ which are inverses of
each other. Equivalent spaces have the same effectively open sets,
the same computable points, and the same computable maps. Moreover,
our definition of Schnorr randomness will be equivalent for equivalent
spaces (as will be evident from Proposition~\ref{prop:iso-preserve-cont}).
For all intents and purposes, we consider two equivalent computable
metric spaces to be the same. Therefore, when we say, e.g., that $\mathbb{X}\times\mathbb{Y}$
is a computable metric space, we mean under any of the standard metrics
that generates the product topology (e.g.~$d_{\mathbb{X}}+d_{\mathbb{Y}}$,
$\sqrt{d_{\mathbb{X}}^{2}+d_{\mathbb{Y}}^{2}}$, or $\max\{d_{\mathbb{X}},d_{\mathbb{Y}}\}$)
and any standard choice of basic points.
\end{rem}
\begin{defn}
Let $\mathbb{X}$ and $\mathbb{Y}$ be computable metric spaces. 

\begin{itemize}
\item A set $U\subseteq\mathbb{X}$ is \emph{effectively open} ($\Sigma_{1}^{0}$)
if $U=\bigcup_{B\in A}B$ where $A$ is a computable set of basic
open balls (i.e.~fix a standard enumeration of all basic open balls,
and consider a computable set $A\subseteq\mathbb{N}$ of the indices).
\item A sequence $(U^{x})_{x\in\mathbb{X}}$ of subsets of $\mathbb{Y}$
is $\Sigma_{1}^{0}[x]$ if 
\[
U^{x}=\{y\in\mathbb{Y}:(x,y)\in U\}
\]
 for some $\Sigma_{1}^{0}$ set $U\subseteq\mathbb{X}\times\mathbb{Y}$.
A set $V\subseteq Y$ is $\Sigma_{1}^{0}[x]$ if $V=U^{x}$ for some
$x$.
\item A set $C\subseteq\mathbb{X}$ is \emph{effectively closed} ($\Pi_{1}^{0}$)
if $\mathbb{X}\smallsetminus C$ is $\Sigma_{1}^{0}$. Similarly define
$\Pi_{1}^{0}[x]$.
\item A function $f\colon A\rightarrow\mathbb{Y}$ (for $A\subseteq\mathbb{X}$)
is \emph{computable} if there is a partial computable map $F\colon{\subseteq{}}\mathbb{N}^{\mathbb{N}}\rightarrow\mathbb{N}^{\mathbb{N}}$
such that if $h$ is a Cauchy name for $x\in A$, then $h\in\operatorname{dom}F$
and $F(h)$ is a Cauchy name for $f(x)$.
\item A function $f\colon\mathbb{X}\rightarrow(-\infty,\infty]$ is \emph{lower semicomputable}
if $f(x)=\sup_{n\in\mathbb{N}}g(n,x)$ for a computable function $g\colon\mathbb{N}\times\mathbb{X}\rightarrow\mathbb{R}$.
(We will often write $g(n,x)$ as $g_{n}(x)$ and call $(g_{n})_{n\in\mathbb{N}}$
a \emph{computable sequence of computable functions}.) Similarly,
define \emph{upper semicomputable} functions $f\colon\mathbb{X}\rightarrow[-\infty,\infty)$.
\end{itemize}
\end{defn}
The basic continuous operations of analysis\textemdash $\min,\max,+,\cdot$,
etc.\textemdash are all computable and the composition of computable
functions is computable. Also a function $f\colon\mathbb{X}\rightarrow\mathbb{R}$
is computable if and only if it is both upper and lower semicomputable.

\subsection{Computable measure theory}

We assume the reader has some basic familiarity with analysis, including
measure theory and probability, e.g.~\cite{Tao:2011kl,Billingsley:1995}.
Much of the material in this section is known and can be found in
Bienvenu, Gács, Hoyrup, Rojas, Shen \cite[\S 2,5,7]{Bienvenu.Gacs.Hoyrup.ea:2011}.

Measure theory is an abstract theory involving set theory and $\sigma$-algebras,
much of which is not obviously compatible with computability theory.
Instead of working with an abstract measure space $(\mathbb{X},\mathcal{F},\mu)$,
it is sufficient for much of work-a-day analysis and probability theory
to restrict our attention to Borel measures on a Polish space. We
always will assume $\mathbb{X}$ is a Polish space and $\mathcal{F}$
is $\mathcal{B}(\mathbb{X})$, the Borel $\sigma$-algebra of $\mathbb{X}$.
Therefore, we write $(\mathbb{X},\mu)$ instead of $(\mathbb{X},\mathcal{B}(\mathbb{X}),\mu)$.
Moreover, we will also assume $\mu$ is \emph{finite}. That is, $\mu(\mathbb{X})<\infty$.
We denote the space of finite Borel measures on $\mathbb{X}$ by $\finitemeas(\mathbb{X})$
and the space of probability measures as $\probmeas(\mathbb{X})$.

To understand the computability of finite Borel measures, it is helpful
to think about it from a functional analysis point of view. Let $\cont_{b}(\mathbb{X})$
denote the space of bounded continuous functions $f\colon\mathbb{X}\rightarrow\mathbb{R}$.
This is a normed vector space under the sup norm $\|f\|=\sup_{x\in\mathbb{X}}|f(x)|$.
Now, every finite Borel measure $\mu$ gives rise to a positive bounded
linear operator $T\colon\cont_{b}(\mathbb{X})\to\mathbb{R}$ given
by $f\mapsto\int_{{}}fd\mu$. (A linear operator $T$ is \emph{bounded}
if $\|Tf\|\leq\|f\|$ and \emph{positive} if $Tf\geq0$ whenever
$f\geq0$.) Conversely, if $T\colon\cont_{b}(\mathbb{X})\to\mathbb{R}$
is a positive bounded linear operator, then it is given by $f\mapsto\int_{{}}fd\mu$
for some measure $\mu$.

However, it is sufficient to only consider the operator on a countable
family of functions. Consider the following enumerable sets of computable
functions. 
\begin{defn}
\label{def:basic-bump}For a computable metric space $\mathbb{X}$,
define $\mathcal{F}_{0}(\mathbb{X})$ to be the set of \emph{basic bump functions}
$f_{x,r,s}\colon\mathbb{X}\rightarrow\mathbb{R}$ (where $x$ is an
basic point, $r,s\in\mathbb{Q}^{+}$, and $r<s$) given by
\[
f_{x,r,s}(y)=\begin{cases}
1 & d_{\mathbb{X}}(x,y)\leq r\\
\frac{s-d_{\mathbb{X}}(x,y)}{s-r} & r<d_{\mathbb{X}}(x,y)\leq s\\
0 & s<d_{\mathbb{X}}(x,y)
\end{cases}.
\]
Define $\mathcal{E}_{b}(\mathbb{X})$ to be $\mathcal{F}_{0}(\mathbb{X})\cup\{\mathbf{1}\}$
closed under $\max$, $\min$, and rational linear combinations. Call
these \emph{basic functions}. Fix a natural enumeration of $\mathcal{E}_{b}(\mathbb{X})$.
Every function in $\mathcal{E}_{b}(\mathbb{X})$ is bounded and its
bounds are computable from the index for the function.
\end{defn}
\begin{defn}
\label{def:Measure metric}If $\mathbb{X}$ is a computable metric
space, then for the spaces $\probmeas(\mathbb{X})$ and $\finitemeas(\mathbb{X})$
we will adopt the following computable metric space structure (see
Kallenberg \cite[15.7.7]{Kallenberg:1983}):

\begin{itemize}
\item The metric is 
\[
d(\mu,\nu)=\sum_{i}2^{-i}\left(1-\exp\left|\int_{{}}f_{i}\,d\mu-\int_{{}}f_{i}\,d\nu\right|\right)
\]
where $(f_{i})_{i\in\mathbb{N}}$ is the enumeration of the basic
functions $\mathcal{E}_{b}(\mathbb{X})$.
\item The basic points are measures of the form $q_{1}\delta_{x_{1}}+\ldots+q_{n}\delta_{x_{n}}$
where $q_{j}\in\mathbb{Q}^{+}$, $x_{j}$ is an basic point of $\mathbb{X}$
(resp.\ $\mathbb{Y}$), and $\delta_{x_{j}}$ is the \emph{Dirac measure}
(the probability measure with unit mass concentrated on $x_{j}$).
For $\probmeas(\mathbb{X})$, we also require $q_{1}+\ldots+q_{n}=1$.
\end{itemize}
\end{defn}
The topologies on these metric spaces are the ones associated with
weak convergence of measures. For $\mu_{n},\nu\in\finitemeas(\mathbb{X})$
or $\probmeas(\mathbb{X})$, $\mu_{n}\rightarrow\nu$ \emph{weakly}
iff $\int_{{}}f\,d\mu_{n}\rightarrow\int_{{}}f\,d\nu$ for all $f\in\cont_{b}(\mathbb{X})$.
(See Kallenberg \cite[\S 15.7]{Kallenberg:1983}.)

Despite the above foundational matters, the only properties we need
of $\probmeas(\mathbb{X})$ and $\finitemeas(\mathbb{X})$ are that
they are computable metric spaces and that they satisfy the following
lemmas.
\begin{lem}
\label{lem:char-comp-meas}Let $\mathbb{A}$ and $\mathbb{X}$ be
computable metric spaces. Any map $a\mapsto\mu_{a}$ of type $\mathbb{A}\rightarrow\finitemeas(\mathbb{X})$
(resp.~$\mathbb{A}\rightarrow\probmeas(\mathbb{X})$) is computable
if and only if $a,i\mapsto\int_{{}}f_{i}(x)\,d\mu_{a}(x)$ is computable
(where $(f_{i})_{i\in\mathbb{N}}$ are the basic functions used in
Definition~\ref{def:Measure metric}).
\end{lem}
\begin{proof}
This follows from Definition~\ref{def:Measure metric} and the fact
that $\int_{{}}f_{i}(x)\,d\nu(x)$ is uniformly computable in $\nu$
and $i$ for basic measures $\nu=q_{1}\delta_{x_{1}}+\ldots+q_{n}\delta_{x_{n}}$.
\end{proof}
This next lemma can be viewed as a computable version of the Portmanteau
Theorem in probability theory.
\begin{lem}[Computable Portmanteau Theorem]
\label{lem:Portmanteau}Let $\mathbb{X}$ be a computable metric
space.

\begin{enumerate}
\item \label{enu:lsc-integral}The map $\mu\mapsto\int_{{}}g(\mu,x)\,d\mu(x)$
is lower (resp.\ upper) semicomputable for any lower semicomputable
function $g\colon\finitemeas(\mathbb{X})\times\mathbb{X}\rightarrow[0,\infty]$
(resp.~upper semicomputable function $g\colon\finitemeas(\mathbb{X})\times\mathbb{X}\rightarrow[-\infty,0]$.)
\item \label{enu:open-integral}The map $\mu\mapsto\mu(U^{\mu})$ is lower
semicomputable for effectively open sets $U\subseteq\finitemeas(\mathbb{X})\times\mathbb{X}$.
\item \label{enu:closed-integral}The map $\mu\mapsto\mu(C^{\mu})$ is upper
semicomputable for effectively closed sets $C\subseteq\finitemeas(\mathbb{X})\times\mathbb{X}$.
\item \label{enu:bounded-cont-integral}The map $\mu\mapsto\int_{{}}f(\mu,x)\,d\mu$
is computable for a bounded computable function $f\colon\finitemeas(\mathbb{X})\times\mathbb{X}\rightarrow\mathbb{R}$
with computable bounds $\mu\mapsto a(\mu),b(\mu)$ such that for all
$x\in\mathbb{X}$, $a(\mu)\leq f(\mu,x)\leq b(\mu)$.
\end{enumerate}
\noindent These results also hold for $\probmeas(\mathbb{X})$ in
place of $\finitemeas(\mathbb{X})$. 
\end{lem}
\begin{proof}
Approximate the integrand or set with basic functions from below ((\ref{enu:lsc-integral})
and (\ref{enu:open-integral})), from above (\ref{enu:closed-integral}),
or from both below and above (\ref{enu:bounded-cont-integral}). Then
apply Lemma~\ref{lem:char-comp-meas}.
\end{proof}
Using the previous two lemmas, one can show that the metric we put
on $\probmeas(\mathbb{X})$ is equivalent to more well-known computable
metrics such as the Wasserstein metric or the Levy-Prokhorov metric.
(See Hoyrup and Rojas \cite[\S 4]{Hoyrup.Rojas:2009}.)

Define the \emph{norm} of a measure $\mu$ as $\|\mu\|:=\int_{{}}1\,d\mu$.
The Computable Portmanteau Theorem tells us that the map $\mu\mapsto\|\mu\|$
is computable.

\section{\label{sec:SR-comp}Schnorr randomness for computable measures}

Martin-Löf randomness \cite{Martin-Lof:1966mz} and Schnorr randomness
\cite{Schnorr:1969ly} were both originally characterized via ``sequential
tests.'' The sequential tests characterizing Martin-Löf randomness
(originally referred to as just ``sequential tests'') are now usually
referred to as Martin-Löf tests. The ones for Schnorr randomness (originally
referred to as ``total sequential tests'') are now usually referred
to as Schnorr tests. We will adopt the more explicit terminology ``Martin-Löf/Schnorr
sequential test.''\footnote{The term ``Martin-Löf test'' is ambiguous, since it is also used
to refer to a sequential test\textemdash even if it is used to characterize
another notion of randomness, e.g.~``bounded Martin-Löf test''
characterizing computable randomness \cite[Def.~7.1.21(ii)]{Downey.Hirschfeldt:2010}.
Alternately ``Martin-Löf test,'' or just ``$\mu$-test,'' is sometimes
used to refer to an author's preferred type of test for Martin-Löf
randomness, even if it is not a sequential test. By using the terms
``Martin-Löf sequential test'' and ``Schnorr sequential test,''
we avoid these ambiguities.}
\begin{defn}
\label{def:Schnorr-comp-meas}Assume $\mu$ is a computable probability
measure on $\mathbb{X}$. A \emph{Schnorr sequential $\mu$-test}
is a sequence $(U_{n})_{n\in\mathbb{N}}$ of open sets where $U_{n}$
is $\Sigma_{1}^{0}[n]$ such that

\begin{enumerate}
\item $\mu(U_{n})\leq2^{-n}$ for all $n$ and
\item $\mu(U_{n})$ is uniformly computable in $n$.
\end{enumerate}
\noindent A point $x\in\mathbb{X}$ is \emph{Schnorr $\mu$-random}
if $x\notin\bigcap_{n}U_{n}$ for every Schnorr sequential $\mu$-test
$(U_{n})_{n\in\mathbb{N}}$.
\end{defn}
Define \emph{Martin-Löf sequential $\mu$-test} and \emph{Martin-Löf $\mu$-randomness}
analogously, except omit condition (2).

Levin \cite[\S3]{Levin:1976uq} introduced a different style of test,
now called an integral test. (Other names include average-bounded
test and expectation-bounded test.) Miyabe \cite{Miyabe:2013uq} gave
the following integral test characterization for Schnorr randomness.
\begin{defn}[{Miyabe \cite[Def.~3.4]{Miyabe:2013uq}}]
\label{def:Schnorr-int-test}Assume $\mu$ is a computable probability
measure on $\mathbb{X}$. A \emph{Schnorr integral $\mu$-test} is
a lower semicomputable function $t\colon\mathbb{X}\rightarrow[0,\infty]$
such that 

\begin{enumerate}
\item $\int_{{}}t\,d\mu<\infty$ and
\item $\int_{{}}t\,d\mu$ is computable.
\end{enumerate}
\end{defn}
\begin{prop}[{Miyabe \cite[Def.~3.5]{Miyabe:2013uq}}]
\label{prop:Miyabe-int-test-Sch}A point $x\in\mathbb{X}$ is Schnorr
$\mu$-random if and only if $t(x)<\infty$ for every Schnorr integral
$\mu$-test $t$.
\end{prop}
Define \emph{Martin-Löf integral $\mu$-test} analogously, except
omit condition (2). Levin \cite{Levin:1976uq} showed that \emph{Martin-Löf integral $\mu$-tests}
characterize Martin-Löf $\mu$-randomness.
\begin{defn}
Assume $\mu$ is a computable probability measure on $\mathbb{X}$.
A \emph{Kurtz $\mu$-test} is an effectively closed set $P\subseteq\mathbb{X}$
such that $\mu(P)=0$. A point $x\in\mathbb{X}$ is \emph{Kurtz $\mu$-random}
if $x\notin P$ for all Kurtz tests $P$.
\end{defn}
It is well known that Martin-Löf $\mu$-random implies Schnorr $\mu$-random
implies Kurtz $\mu$-random and that for many (but not all) measures
$\mu$ these implications are strict. For more background on randomness,
the standard books are \cite{Downey.Hirschfeldt:2010,Nies:2009,Li.Vitanyi:2008}.
However, these books take place in the setting of computable measures
on Cantor space. For a more general setting, closer to our own, we
recommend the comprehensive paper by Bienvenu, Gács, Hoyrup, Rojas,
and Shen \cite{Bienvenu.Gacs.Hoyrup.ea:2011} and the lecture notes
by Gács \cite{Gacs:aa}.

\section{\label{sec:SR-noncomp}Schnorr randomness for noncomputable measures}

Levin \cite{Levin:1976uq} extended Martin-Löf randomness (on Cantor
space) to noncomputable probability measures by using a uniform test\textemdash that
is a single test which combines tests for every measure. However,
the uniform tests we use are a modification due to Gács \cite{Gacs:2005}.
Also, Gács \cite{Gacs:2005} and Hoyrup and Rojas \cite{Hoyrup.Rojas:2009}
generalized uniform tests to all computable metric spaces. (See also
Bienvenu, Gács, Hoyrup, Rojas, and Shen \cite{Bienvenu.Gacs.Hoyrup.ea:2011}.)
Moreover, uniform tests can also be used to define Martin-Löf randomness
relative to noncomputable oracles.

Our definition of Schnorr randomness for noncomputable measures and
noncomputable oracles follows this approach. (See Theorems~\ref{thm:SR-over-K},
\ref{thm:SR-name}, and \ref{thm:SR-seq-test} for other useful characterizations.)
\begin{defn}
\label{def:SR-noncomp-meas}Given computable metric spaces $\mathbb{A}$
and $\mathbb{X}$, a \emph{uniform Schnorr integral test} is a lower
semicomputable function $t\colon\mathbb{A}\times\finitemeas(\mathbb{X})\times\mathbb{X}\rightarrow[0,\infty]$
such that

\begin{enumerate}
\item $\int_{{}}t(a,\mu,x)\,d\mu(x)<\infty$ for all $\mu\in\finitemeas(\mathbb{X})$
and $a\in\mathbb{A}$ and 
\item the map $a,\mu\mapsto\int_{{}}t(a,\mu,x)\,d\mu(x)$ is a computable
map of type $\mathbb{A}\times\finitemeas(\mathbb{X})\to\mathbb{R}$.
\end{enumerate}
\noindent Given $\mu\in\finitemeas(\mathbb{X})$ and $a\in\mathbb{A}$,
a point $x\in\mathbb{X}$ is \emph{uniformly Schnorr $\mu$-random relative to the oracle $a$}
($x\in\SR{\mu}a$) if $t(a,\mu,x)<\infty$ for every Schnorr integral
test $t$. (If the oracle $a$ is computable, we just write $\SR{\mu}{}$.)
\end{defn}
The concepts of \emph{uniform Martin-Löf integral test} and \emph{Martin-Löf $\mu$-randomness relative to the oracle $a$}
($\MLR{\mu}a$) are defined analogously, except omit condition (2).

This definition of a uniform Kurtz $\mu$-randomness is also new.
\begin{defn}
Say that a \emph{uniform Kurtz test} is an effectively closed set
$P\subseteq\mathbb{A}\times\finitemeas(\mathbb{X})\times\mathbb{X}$
such that for all $a\in\mathbb{A}$ and $\mu\in\finitemeas(\mathbb{X})$,
the set $P_{\mu}^{a}:=\{x\in\mathbb{X}:(a,\mu,x)\in P\}$ has $\mu$-measure
0. Say that $x_{0}\in\mathbb{X}$ is \emph{uniformly Kurtz $\mu_{0}$-random relative to the oracle $a_{0}$}
($x_{0}\in\KR{\mu_{0}}{a_{0}}$) if $x_{0}\notin P_{\mu_{0}}^{a_{0}}$
for all uniform Kurtz tests $P$.
\end{defn}
\begin{rem}
We say $x$ is ``\emph{uniformly} Schnorr $\mu$-random relative
to the oracle $a$'' to emphasize that we are using a test which
is uniform in both the measure $\mu$ and the oracle $a$. This terminology
is adapted from Miyabe and Rute \cite{Miyabe.Rute:2013} where they
discuss Schnorr randomness uniformly relative to an oracle (previously
studied by Franklin and Stephan \cite{Franklin.Stephan:2010} under
the name truth-table Schnorr randomness). (See Definition~\ref{def:sch-unif-rel}.)
There is also a non-uniform version of Schnorr randomness for noncomputable
measures which we briefly discuss in Subsection~\ref{subsec:naive-def},
but we believe the uniform definition is the correct one. For Martin-Löf
randomness, there is no distinction between the uniform and nonuniform
definitions.
\end{rem}
\begin{notation}
Rather then writing $t(a,\mu,x)$, we will write the more compact
$t_{\mu}^{a}(x)$. That is, the superscript is the oracle, the subscript
is the measure, and the argument is the point being tested. The same
goes for the notation $x\in\SR{\mu}a$. To match our new notation
$t_{\mu}^{a}(x)$, we will write (and think about) a uniform Schnorr
integral test as a computable family of tests $\{t_{\mu}^{a}\}_{a\in\mathbb{A},\mu\in\finitemeas(\mathbb{X})}$.

Another convention is that when we write multiple oracles separated
by commas, e.g. $x\in\mathsf{SR}_{\mu}^{a,b}$, we mean that $x$
is Schnorr $\mu$-random uniformly relative to the pair $(a,b)\in\mathbb{A}\times\mathbb{B}$.
(This will be justified by Proposition~\ref{prop:reducible-oracles}.)
Both these conventions will help when we use measures as oracles later.
\end{notation}
We give a few basic results here, however, most must wait until we
have developed some tools.
\begin{prop}
\label{prop:MLR-to-SR}If $x\in\MLR{\mu}a$, then $x\in\SR{\mu}a$.
\end{prop}
\begin{proof}
A uniform Schnorr integral test is a uniform Martin-Löf integral test.
\end{proof}
We delay showing that $\SR{\mu}a\subseteq\KR{\mu}a$ (Proposition~\ref{prop:SR-implies-KR})
until we have more tools.

The one special measure in $\finitemeas(\mathbb{X})$ is the zero
measure $\mu=0$ such that $\mu(A)=0$ for all measurable sets $A$.
\begin{prop}
\label{prop:.zero-meas}$\SR 0a=\varnothing$.
\end{prop}
\begin{proof}
Let $t_{\mu}^{a}(x)=\|\mu\|^{-1/2}$. (Since $\|\mu\|^{1/2}$ is computable
by Lemma~\ref{lem:Portmanteau} and nonnegative, its inverse is lower
semicomputable, even at $\mu=0$.) Then $\int_{{}}t_{\mu}^{a}(x)\,d\mu=\|\mu\|^{1/2}$
and for all $x\in\mathbb{X}$, $t_{0}^{a}(x)=\infty$. So $x\notin\SR 0a$.
\end{proof}

\section{Tests restricted to closed sets of measures\label{sec:tests-restricted}}

In order to prove even the most basic properties about our new definition
of Schnorr randomness for noncomputable measures, we will need a more
flexible notion of uniform integral test.
\begin{defn}
\label{def:SR-unif-int-test-K}Let $\mathbb{A}$ and $\mathbb{X}$
be computable metric spaces and let $K\subseteq\mathbb{A}\times\finitemeas(\mathbb{X})$
be effectively closed (i.e.\ $\Pi_{1}^{0}$). A \emph{uniform Schnorr integral test restricted to $K$}
is a computably indexed family of lower semicomputable functions $\{t_{\mu}^{a}\}_{(a,\mu)\in K}$
(i.e.\ $t\colon\mathbb{X}\times K\rightarrow\mathbb{R}$ is lower
semicomputable) such that 

\begin{enumerate}
\item $\int_{{}}t_{\mu}^{a}(x)\,d\mu(x)<\infty$ for all $(a,\mu)\in K$
and 
\item the map $a,\mu\mapsto\int_{{}}t_{\mu}^{a}(x)\,d\mu(x)$ is a computable
map of type $K\rightarrow\mathbb{R}$.
\end{enumerate}
\end{defn}
\begin{rem}
\label{rem:normalized}If there are no pairs $(a,\mu)\in K$ where
$\mu$ is the zero measure $0$, we can normalize our uniform integral
test $\{t_{\mu}^{a}\}_{(a,\mu)\in K}$ by replacing it with 
\[
s_{\mu}^{a}=\frac{(t_{\mu}^{a}+1)}{\int(t_{\mu}^{a}+1)\,d\mu}.
\]
Then, $\int_{{}}s_{\mu}^{a}\,d\mu=1$ for all $(a,\mu)\in K$ and
$s$ is bounded below by a computable function $g(a,\mu,x)>0$. Call
such an integral test \emph{normalized}.
\end{rem}
The goal of this section is to prove that these more general uniform
Schnorr integral tests characterize Schnorr randomness for noncomputable
measures and oracles.
\begin{thm}
\label{thm:SR-over-K}Let $\mathbb{A}$ and $\mathbb{X}$ be computable
metric spaces and let $K\subseteq\mathbb{A}\times\finitemeas(\mathbb{X})$
be effectively closed. For all $(a,\mu)\in K$, $x\in\SR{\mu}a$ if
and only if $t_{\mu}^{a}(x)<\infty$ for all uniform Schnorr integral
tests $\{t_{\mu}^{a}\}_{(a,\mu)\in K}$ restricted to $K$ (as in
Definition~\ref{def:SR-unif-int-test-K}).
\end{thm}
We will give a proof of Theorem~\ref{thm:SR-over-K} below after
providing two lemmas in computable analysis. Before we give those
lemmas let us provide a proof sketch of the left-to-right direction
of Theorem~\ref{thm:SR-over-K}. Assume $t_{\mu_{0}}^{a_{0}}(x_{0})=\infty$
for some uniform Schnorr integral test $\{t_{\mu}^{a}\}_{(a,\mu)\in K}$
restricted to $K$ such that $\int_{{}}t_{\mu}^{a}(x)\,d\mu(x)=1$.
It suffices to extend $\{t_{\mu}^{a}\}_{(a,\mu)\in K}$ to a uniform
Schnorr integral test $\{\bar{t}_{\mu}^{a}\}_{a\in\mathbb{A},\mu\in\finitemeas(\mathbb{X})}$
such that $\bar{t}_{\mu}^{a}(x_{0})=\infty$. Then $x_{0}\notin\SR{\mu_{0}}{a_{0}}$. 

It is easy to extend $t_{\mu}^{a}$ to a lower semicomputable function
$\bar{t}_{\mu}^{a}$ outside of $K$, but it is more challenging to
guarantee that $a,\mu\mapsto\int_{{}}\bar{t}_{\mu}^{a}(x)\,d\mu(x)$
remains computable. To do this, we can break up $t_{\mu}^{a}$ into
a supremum of bounded continuous functions $\{f_{\mu}^{n;a}\}_{n\in\mathbb{N},(a,\mu)\in K}$
which are increasing in $n$ such that $t_{\mu}^{a}=\sup_{n}f_{\mu}^{n;a}$
and $\int_{{}}f_{\mu}^{n;a}(x)\,d\mu(x)=1-2^{-n}$. Using a computable
version of the Tietze Extension Theorem (Lemma~\ref{lem:Tietze}
below) we can naturally extend $f_{\mu}^{n;a}$ to values of $(a,\mu)$
outside $K$. This, however, may change the integral of $f_{\mu}^{n;a}$.
Since $f_{\mu}^{n;a}$ is bounded and continuous, its integral is
computable and we can rescale the extension of $f_{\mu}^{n;a}$ to
get a new function $g_{\mu}^{n;a}$ with the desired integral. Finally,
we construct the extension of $t_{\mu}^{a}$ to be $\bar{t}_{\mu}^{a}=\sup_{n}g_{\mu}^{n;a}$.
We have that $\int_{{}}\bar{t}_{\mu}^{a}(x)\,d\mu(x)=1$ as desired.

The actually proof is more complicated for two reasons. Firstly, when
rescaling we need to avoid division by zero. Secondly, it is nontrivial
to find the functions $f_{\mu}^{n;a}$ described above. To do this
we need to approximate $t_{\mu}^{a}$ from below not with a \emph{discrete}
sequence of functions $f_{\mu}^{n;a}$, but instead with a \emph{continuous}
family $f_{\mu}^{r;a}$ ($r\in[0,1)$) such that $t_{\mu}^{a}=\sup_{n}f_{\mu}^{r;a}$
and $\int_{{}}f_{\mu}^{r;a}(x)\,d\mu(x)=r$. This next lemma shows
that such a continuous approximation is computable.
\begin{lem}
\label{lem:approx-below}Given a \emph{uniform Schnorr integral test}
$\{t_{\mu}^{a}\}_{(a,\mu)\in K}$, there is a computable function
$f\colon[0,\infty)\times K\times\mathbb{X}\rightarrow[0,\infty)$
nondecreasing in the first coordinate such that $f(0,a,\mu,x)=0$,
\[
t_{\mu}^{a}(x)=\sup\{f(r,a,\mu,x):0\leq r<\infty\},
\]
and $r,a,\mu\mapsto\int_{{}}f(r,a,\mu,x)\,d\mu(x)$ is uniformly computable.
Moreover, if there is a computable function $g\colon K\times\mathbb{X}\rightarrow[0,\infty)$
such that $t_{\mu}^{a}(x)\geq g(a,\mu,x)>0$, then $f$ is strictly
increasing in the first coordinate. (We will often write $f(r,a,\mu,x)$
as $f_{\mu}^{r;a}(x)$.)
\end{lem}
\begin{proof}
Let $h\colon\mathbb{N}\times\mathbb{A}\times\finitemeas(\mathbb{X})\times\mathbb{X}\rightarrow[0,\infty)$
be a computable function such that $h(n,a,\mu,x)$ is nondecreasing
in $n$ and 
\[
t_{\mu}^{a}(x)=\sup_{n}h(n,a,\mu,x).
\]
We may assume $h(n,a,\mu,x)\leq n$ by replacing it with $\min\{n,h(n,a,\mu,x)\}$.

Now define $f$ by interpolation (where $\lfloor r\rfloor$ denotes
the greatest integer $\leq r$), 
\[
f(r,a,\mu,x)=h(\lfloor r\rfloor,a,\mu,x)+(r-\lfloor r\rfloor)\left(h(\lfloor r\rfloor+1,a,\mu,x)-h(\lfloor r\rfloor,a,\mu,x)\right).
\]
Since $f(r,a,\mu,x)$ is computable and 
\[
f(r,a,\mu,x)\leq h(\lfloor r\rfloor+1,a,\mu,x)\leq r+1
\]
we can uniformly compute $\int_{{}}f(r,a,\mu,x)\,d\mu(x)$ from $r,a,\mu$
(Lemma~\ref{lem:Portmanteau}).

If $t_{\mu}^{a}(x)\geq g(a,\mu,x)>0$, then $t_{\mu}^{a}(x)-g(a,\mu,x)$
is lower semicomputable. By Lemma~\ref{lem:Portmanteau}, $\int_{{}}t_{\mu}^{a}(x)-g(a,\mu,x)\,d\mu$
and $\int_{{}}g(a,\mu,x)\,d\mu$ are both lower semicomputable in
$(a,\mu)\in K$. Moreover, these integrals are computable since they
are lower semicomputable and their sum is computable. Therefore, there
is some $f\colon[0,\infty)\times K\times\mathbb{X}\rightarrow[0,\infty)$
(as defined earlier in this proof) which approximates $t_{\mu}^{a}(x)-g(a,\mu,x)$
from below. Add $(1-2^{-r})g(a,\mu,x)$ to $f(r,a,\mu,x)$ to get
a strictly increasing approximation of $t_{\mu}^{a}(x)$.
\end{proof}
The Tietze Extension Theorem is an important tool in analysis. The
following is a computable version.
\begin{lem}[{Computable Tietze Extension Theorem \cite[Thm.~19]{Weihrauch:2001}}]
\label{lem:Tietze}If $\mathbb{X}$ is a computable metric space,
$K\subseteq\mathbb{X}$ is effectively closed, and $f\colon K\rightarrow\mathbb{R}$
is computable, then we can effectively extend $f$ to a computable
function $\bar{f}\colon\mathbb{X}\rightarrow\mathbb{R}$ such that
$\sup_{x\in\mathbb{X}}\bar{f}(x)=\sup_{x\in K}f(x)$ and $\inf_{x\in\mathbb{X}}\bar{f}(x)=\inf_{x\in K}f(x)$.
\end{lem}
Now we prove Theorem~\ref{thm:SR-over-K}.
\begin{proof}[Proof of Theorem~\ref{thm:SR-over-K}]
Fix $x_{0}\in\mathbb{X}$ and $(a_{0},\mu_{0})\in K$.

($\Leftarrow$) Assume $x_{0}\notin\SR{\mu_{0}}{a_{0}}$. Then there
is some uniform Schnorr integral test $\{t_{\mu}^{a}\}_{a\in\mathbb{A},\mu\in\finitemeas(\mathbb{X})}$
(as in Definition~\ref{def:SR-noncomp-meas}) such that $t_{\mu_{0}}^{a_{0}}(x)=\infty$.
We have that $\{t_{\mu}^{a}\}_{(a,\mu)\in K}$ is a uniform Schnorr
integral test restricted to $K$.

($\Rightarrow$) Assume $t_{\mu_{0}}^{a_{0}}(x_{0})=\infty$ for some
uniform Schnorr integral test $\{t_{\mu}^{a}\}_{(a,\mu)\in K}$ restricted
to $K$. By Proposition~\ref{prop:.zero-meas}, $\SR 0a=\varnothing$.
Therefore, we assume $K$ does not contain the $0$ measure. (Just
replace $K$ with $K\smallsetminus(\mathbb{A}\times B)$ where $B\subseteq\finitemeas(\mathbb{X})$
is some sufficiently small basic open ball containing $0$.) Then
by Remark~\ref{rem:normalized}, we may assume $t_{\mu}^{a}$ is
normalized and bounded below by a nonnegative computable function.

Our basic goal is to extend $\{t_{\mu}^{a}\}_{(a,\mu)\in K}$ to a
uniform Schnorr integral test $\{\bar{t}_{\mu}^{a}\}_{a\in\mathbb{A},\mu\in\finitemeas(\mathbb{X})}$
such that $\bar{t}_{\mu}^{a}(x_{0})=\infty$. (For technical reasons,
the uniform Schnorr integral test $\{\bar{s}_{\mu}^{a}\}_{a\in\mathbb{A},\mu\in\finitemeas(\mathbb{X})}$
we construct at the end of this proof will be an extension of $\|\mu\|t_{\mu}^{a}(x)+1$.)
By Lemma~\ref{lem:approx-below}, there is a computable family $\{f_{\mu}^{r;a}\}_{r\in[0,\infty),(a,\mu)\in K}$
of strictly increasing uniformly computable functions such that for
$(a,\mu)\in K$ we have $t_{\mu}^{a}=\sup_{r}f_{\mu}^{r;a}$ and $r,a,\mu\mapsto\int_{{}}f_{\mu}^{r;a}\,d\mu$
is computable. 

Notice the integral function $I(r,a,\mu)=\int_{{}}f_{\mu}^{r;a}\,d\mu$
is a computable bijection in the first coordinate from $[0,\infty)$
to $[0,1)$. (Indeed, $\int_{{}}f_{\mu}^{r;a}\,d\mu$ is strictly
increasing in $r$, $f_{\mu}^{0;a}=0$, and by the Monotone Convergence
Theorem, 
\[
\sup_{r}\int_{{}}f_{\mu}^{r;a}\,d\mu=\int_{{}}\sup_{r}f_{\mu}^{r;a}\,d\mu=\int_{{}}t_{\mu}^{a}\,d\mu=1.
\]
So it is a bijection.) Now, the inverse map $I^{-1}(r,a,\mu)$ such
that $I(I^{-1}(r,a,\mu),a,\mu)=r$ for $r\in[0,1)$ is also computable.
So we consider the function $g\colon[0,1)\times K\times\mathbb{X}\rightarrow[0,\infty)$
given by
\[
g(r,a,\mu,x)=f(I^{-1}(r,a,\mu),a,\mu,x).
\]
Then $t_{\mu}^{a}=\sup\{g(r,a,\mu,x):r\in[0,1)\}$ and $\int_{{}}g(r,a,\mu,x)\,d\mu(x)=r$.

To avoid division by zero, instead of $t_{\mu}^{a}$ and $g(r,a,\mu,x)$
we will consider the consider the lower semicomputable function $s_{\mu}^{a}=\|\mu\|t_{\mu}^{a}+1$
and the computable function 
\[
h(r,a,\mu,x)=\|\mu\|g(r,a,\mu,x)+r.
\]
So $s_{\mu}^{a}=\sup\{h(r,a,\mu,x):r\in[0,1)\}$ and $\int_{{}}h(r,a,\mu,x)\,d\mu=2\|\mu\|r$.
We still have that $s_{\mu_{0}}^{a_{0}}(x_{0})=\infty$. Then use
the Computable Tietze Extension Theorem (Lemma~\ref{lem:Tietze})
to extend $h\colon[0,1)\times K\times\mathbb{X}\rightarrow[0,\infty)$
to $\bar{h}\colon[0,1)\times\mathbb{A}\times\finitemeas(\mathbb{X})\times\mathbb{X}\rightarrow[0,\infty)$
such that $\bar{h}(r,a,\mu,x)\geq r$. (Technically, one applies the
Computable Tietze Extension Theorem to $h(r,a,\mu,x)-r$ to get a
nonnegative extension. Then add $r$.) Last, normalize $\bar{h}(r,a,\mu,x)$
to get 
\begin{equation}
\hat{h}(r,a,\mu,x)=\frac{2\|\mu\|r}{\int_{{}}\bar{h}(r,a,\mu,x)\,d\mu}\bar{h}(r,a,\mu,x)\label{eq:normal-approx}
\end{equation}
which is equal to $h(r,a,\mu,x)$ for $(a,\mu)\in K$.\footnote{As a technical detail we need that $\hat{h}$ is still continuous
and computable at $\mu=0$. To avoid the issues of dividing by $0$
in (\ref{eq:normal-approx}), do the following. By our assumption
above, there is a basic closed ball $\overline{B}\subseteq\finitemeas(\mathbb{X})$
containing $0$ such that $K\cap(\mathbb{A}\times\overline{B})=\varnothing$.
Before applying the Computable Tietze Extension Theorem, define $h$
for $\mu\in\overline{B}$ to be $h(r,a,\mu,x)=2r$. Then apply the
Computable Tietze Extension Theorem (Lemma~\ref{lem:Tietze}) to
$h\colon[0,1)\times(K\cup\overline{B})\times\mathbb{X}\rightarrow[0,\infty)$
to get $\bar{h}$. In the neighborhood around $\mu=0$, $\bar{h}(r,a,\mu,x)=2r$
and $\hat{h}(r,a,\mu,x)=2r$ by (\ref{eq:normal-approx}). We explicitly
set $\hat{h}(r,a,0,x)=2r$. This ensures that $\hat{h}$ is computable,
even at $\mu=0$.} Set $\{\bar{s}_{\mu}^{a}\}_{a\in\mathbb{A},\mu\in\finitemeas(\mathbb{X})}$
to be 
\[
\bar{s}_{\mu}^{a}=\sup\{\bar{h}(r,a,\mu,x):r\in[0,1)\}.
\]
This extends $\{s_{\mu}^{a}\}_{(a,\mu)\in K}$. So $\bar{s}_{\mu_{0}}^{a_{0}}(x_{0})=s_{\mu_{0}}^{a_{0}}(x_{0})=\|\mu_{0}\|t_{\mu_{0}}^{a_{0}}(x_{0})+1=\infty$.
\end{proof}

\section{Basic properties of Schnorr randomness for noncomputable measures\label{sec:basic-results}}

Armed with Theorem~\ref{thm:SR-over-K}, it is easy to show some
basic results about Schnorr randomness for noncomputable measures.
In this section, $\mathbb{A}$ and $\mathbb{X}$ will denote computable
metric spaces, $a$ and $x$ will denote points in $\mathbb{A}$ and
$\mathbb{X}$ respectively (where $a$ is used as an oracle), and
$\mu$ will denote a measure in $\finitemeas(\mathbb{X})$.

Firstly, and most importantly, our new definition of Schnorr randomness
agrees with Schnorr randomness for computable probability measures.
\begin{prop}
\label{prop:our-def-extends-usual-def}If $\mu\in\probmeas(\mathbb{X})$
and $a\in\mathbb{A}$ are computable, then $x\in\SR{\mu}a$ if and
only if $x$ is Schnorr $\mu$-random (in the original sense of Definition~\ref{def:Schnorr-comp-meas}).
\end{prop}
\begin{proof}
Let $K$ be the effectively closed singleton set $\{(a,\mu)\}$. A
uniform Schnorr integral test over $K$ (Definition~\ref{def:SR-unif-int-test-K})
is the same as a Schnorr integral $\mu$-test (Definition~\ref{def:Schnorr-int-test}).
Now apply Theorem~\ref{thm:SR-over-K} to Miyabe's Schnorr integral
test characterization (Proposition~\ref{prop:Miyabe-int-test-Sch})
of Schnorr $\mu$-randomness.
\end{proof}
Notice, in Definition~\ref{def:SR-noncomp-meas} that we defined
the uniform Schnorr integral $\mu$-tests to be of type $t\colon\mathbb{A}\times\finitemeas(\mathbb{X})\times\mathbb{X}\to[0,\infty)$.
This next proposition says that we could we can replace $\finitemeas(\mathbb{X})$
in the definition with $\probmeas(\mathbb{X})$ to get the same definition
of randomness.
\begin{prop}
If $\mu$ is a probability measure, then $\SR{\mu}{}$ (using $\probmeas(\mathbb{X})$)
is the same as $\SR{\mu}{}$ (using $\finitemeas(\mathbb{X})$).
\end{prop}
\begin{proof}
Note that $\probmeas(\mathbb{X})$ is an effectively closed subspace
of $\finitemeas(\mathbb{X})$. Now apply Theorem~\ref{thm:SR-over-K}.
\end{proof}
While Schnorr randomness for noncomputable measures is new, Schnorr
randomness for noncomputable oracles has already been given a lot
of attention\textemdash at least when the oracles are in $\{0,1\}^{\mathbb{N}}$
and the measure is the fair coin measure $\lambda$. In particular,
our definition agrees with that of Franklin and Stephan \cite[Def.~2.3]{Franklin.Stephan:2010}
(although we use an integral test characterization due to Miyabe and
Rute \cite[Prop.~4.5]{Miyabe.Rute:2013}.)
\begin{defn}
\label{def:sch-unif-rel}Let $x_{0},a_{0}\in\{0,1\}^{\mathbb{N}}$.
Say that $x_{0}$ is \emph{Schnorr $\lambda$-random uniformly relative to $a_{0}$}
(or $x_{0}$ is \emph{$a_{0}$ truth-table Schnorr $\lambda$-random})
if and only if $t^{a_{0}}(x)<\infty$ for all uniformly computable
families $\{t^{a}\}_{a\in\{0,1\}^{\mathbb{N}}}$ of lower semicomputable
functions where $\int_{{}}t^{a}\,d\lambda=1$ for all $a$.
\end{defn}
\begin{prop}
For $x,a\in\{0,1\}^{\mathbb{N}}$, $x\in\SR{\lambda}a$ if and only
if $x$ is Schnorr random uniformly relative to $a$ as in Definition~\ref{def:sch-unif-rel}.
\end{prop}
\begin{proof}
Apply Theorem~\ref{thm:SR-over-K} where $K$ is the effectively
closed set $\{(a,\mu):a\in\{0,1\}^{\mathbb{N}},\mu=\lambda\}$.
\end{proof}
One would like to say that an oracle $b$ is more powerful than an
oracle $a$ if ``$b$ computes $a$.'' However, for Schnorr randomness,
this does not work.\footnote{There are $x,y\in\{0,1\}^{\mathbb{N}}$ such that $x\equiv_{T}y$
and $x\in\SR{\lambda}y$ (see Miyabe and Rute \cite[Prop~6.4]{Miyabe.Rute:2013}).
Since $x\notin\SR{\lambda}x$, this implies that $\SR{\lambda}y\not\subseteq\SR{\lambda}x$
even though $y$ computes $x$.}  Instead, we use the following uniform reducibility.
\begin{defn}
\label{def:uniform-comp}Let $\mathbb{A}$ and $\mathbb{B}$ be computable
metric spaces. Say that $a\in\mathbb{A}$ \emph{uniformly computes}
$b\in\mathbb{B}$ iff there is an effectively closed set $K\subseteq\mathbb{A}$
and a computable map $f\colon K\rightarrow\mathbb{A}$ such that $f(a)=b$.
Say $a$ and $b$ are \emph{uniformly equicomputable} if each uniformly
computes the other.
\end{defn}
It is easy to see that when $\mathbb{A}=2^{\mathbb{N}}$ and $\mathbb{B}\in\{2^{\mathbb{N}},\mathbb{N}^{\mathbb{N}}\}$
then $a$ uniformly computes $b$ if and only if $a$ truth-table
computes $b$. Our definition of uniform computation is well-suited
to Schnorr randomness as this next proposition shows.
\begin{prop}
\label{prop:reducible-oracles}Let $\mu\in\finitemeas(\mathbb{X})$,
$a\in\mathbb{A}$, and $b\in\mathbb{B}$. Assume $(a,\mu)$ uniformly
computes $b$ as in the previous definition. Then $x\in\SR{\mu}a$
implies $x\in\SR{\mu}b$.
\end{prop}
\begin{proof}
Fix $x_{0}$, $\mu_{0}$, $a_{0}$, and $b_{0}$. Assume $x_{0}\notin\SR{\mu_{0}}{b_{0}}$.
Then there is a uniform Schnorr integral test $\{t_{\mu}^{b}\}_{b\in\mathbb{B},\mu\in\finitemeas(\mathbb{X})}$
such that $t_{\mu_{0}}^{b_{0}}(x_{0})=\infty$. Also assume $(a_{0},\mu_{0})$
uniformly computes $b_{0}$, so there is an effectively closed set
$K\subseteq\mathbb{A}\times\finitemeas(\mathbb{X})$ and a computable
map $f\colon K\rightarrow\mathbb{B}$ such that $f(a_{0},\mu_{0})=b_{0}$.
Define a uniform Schnorr integral test $\{s_{\mu}^{a}\}_{(a,\mu)\in K}$
given by $s_{\mu}^{a}=t_{\mu}^{f(a,\mu)}$ for all $(a,\mu)\in K$.
Then 
\[
s_{\mu_{0}}^{a_{0}}(x_{0})=t_{\mu_{0}}^{b_{0}}(x_{0})=\infty.
\]
By Theorem~\ref{thm:SR-over-K}, $x_{0}\notin\SR{\mu_{0}}{a_{0}}$.
\end{proof}
We conjecture that the converse to Proposition~\ref{prop:reducible-oracles}
also holds, showing that uniform computability (as in Definition~\ref{def:uniform-comp})
is precisely the reducibility arising from uniform Schnorr randomness.
Specifically, one would need to answer this next question positively.
\begin{question}
\label{que:reducibility-needed}Assume $a\in\mathbb{A}$ does not
uniformly compute $b\in\mathbb{B}$ (as in Definition~\ref{def:uniform-comp}).
Is there a finite Borel measure $\mu$ on a space $\mathbb{X}$ such
that $\SR{\mu}a\not\subseteq\SR{\mu}b$?
\end{question}
Recall the \emph{support} of $\mu$, denoted $\supp\mu$, is the
smallest closed set of full $\mu$-measure. Equivalently, $\supp\mu=\{x\in\mathbb{X}:\forall r>0\ \mu(B(x,r))>0\}.$
\begin{prop}
\label{prop:support}If $x\in\SR{\mu}{}$ (even $x\in\KR{\mu}{}$),
then $x\in\supp\mu$.
\end{prop}
\begin{proof}
Fix $\mu_{0},x_{0}$. Assume $x_{0}\notin\supp\mu_{0}$. Then there
is some basic open ball $B\subseteq\mathbb{X}$ around $x_{0}$ such
that $\mu_{0}(B)=0$. This ball is effectively open. Let $K$ be the
effectively closed set $\{\mu:\mu(B)=0\}$ (which is effectively closed
since $\mu\mapsto\mu(B)$ is lower semicomputable by Lemma~\ref{lem:Portmanteau}).
Let $\{t_{\mu}\}_{\mu\in K}$ be the uniform Schnorr integral test
on $K$ given by 
\[
t_{\mu}(x)=\begin{cases}
\infty & x\in B\\
0 & \text{otherwise}
\end{cases}.
\]
Then $t_{\mu}$ is lower semicomputable and $\int_{{}}t_{\mu}\,d\mu=0$
for all $\mu\in K$, but $t_{\mu_{0}}(x_{0})=\infty$. Hence $x_{0}\notin\SR{\mu_{0}}{}$.
\end{proof}
\begin{prop}
Let $\mathbb{X}$ be a computable metric space and assume $\mu\in\finitemeas(\mathbb{X})$.
Then $x\in\SR{\mu}{}$ if and only if $\mu\neq0$ and $x\in\SR{\mu/\|\mu\|}{\|\mu\|}$.
\end{prop}
\begin{proof}
We have already handled the zero measure in Proposition~\ref{prop:.zero-meas}.
The rest follows from Proposition~\ref{prop:reducible-oracles} since
$\mu$ and $(\|\mu\|,\mu/\|\mu\|)$ are uniformly equicomputable as
in Definition~\ref{def:uniform-comp}.
\end{proof}
By this last proposition, there is not much loss in restricting our
attention to probability measures, which we will mostly do in the
rest of the paper.

\section{Schnorr randomness on noncomputable product measures\label{sec:SR-product-meas}}

An important result in algorithmic randomness is Van Lambalgen's Theorem,
which characterizes when a pair $(x,y)$ can be random for a product
measure $\mu\otimes\nu$. In this section, let $\mu$ and $\nu$ be
probability measures on $\mathbb{X}$ and $\mathbb{Y}$, respectively.
Since $\probmeas(\mathbb{X})$ and $\probmeas(\mathbb{Y})$ are computable
metric spaces, we can also treat $\mu$ and $\nu$ as oracles. 

Recall that $\mu\otimes\nu$ denotes the \emph{product measure} on
$\mathbb{X}\times\mathbb{Y}$ characterized by the equation 
\[
(\mu\otimes\nu)(A\times B)=\mu(A)\times\nu(B)\qquad(A\subseteq\mathbb{X},B\subseteq\mathbb{Y}).
\]
Every product measure satisfies Fubini's Theorem, which states that
for a $\mu\otimes\nu$-integrable function $f$,
\[
\iint_{\mathbb{X}\times\mathbb{Y}}f(x,y)\,d(\mu\otimes\nu)(x,y)=\int_{\mathbb{X}}\left(\int_{\mathbb{Y}}f(x,y)\,d\nu(y)\right)d\mu(x)=\int_{\mathbb{Y}}\left(\int_{\mathbb{X}}f(x,y)\,d\mu(x)\right)d\nu(y).
\]

Van Lambalgen's Theorem (generalized to noncomputable measures) is
as follows. The main proof goes back to Van Lambalgen \cite[Thm.~5.10]{Lambalgen:1990}.
This extension to noncomputable measures follows the same general
proof. Details can be found in Day and Reimann \cite[Thm.~1.7]{Day.Reimann:2014}
(for $\mu=\nu$ and $\mathbb{X}=\{0,1\}^{\mathbb{N}}$), and Simpson
and Stephan \cite[Thm.~4.11]{Simpson.Stephan:2015} (for $\mathbb{X}=\{0,1\}^{\mathbb{N}}$).
(The same proof can be easily adapted to computable metric spaces.)
\begin{thm}[Van Lambalgen]
For (noncomputable) probability measures $\mu$ and $\nu$,
\[
(x,y)\in\MLR{\mu\otimes\nu}{}\quad\text{iff}\quad x\in\MLR{\mu}{\nu}\text{ and }y\in\MLR{\nu}{\mu,x}.
\]
\end{thm}
We will show that Van Lambalgen's Theorem holds for Schnorr randomness
(as defined in Definition~\ref{def:SR-noncomp-meas}).\footnote{We caution the reader that other papers and books say that Van Lambalgen's
theorem \emph{does not} hold for Schnorr randomness. That is because
those papers and books use the non-uniformly relativized version which
we discuss more in subsection~\ref{subsec:naive-def}.} Our proof generalizes that in Miyabe \cite{Miyabe:2011} and in Miyabe
and Rute \cite{Miyabe.Rute:2013} (both of which concern Van Lambalgen's
Theorem for Schnorr randomness with respect to the fair coin measure
$\lambda$ on $\{0,1\}^{\mathbb{N}}$). Before proving the theorem,
we need this analytic lemma. (The reader familiar with measure theory
will note this is a computable version of a special case of Luzin's
Theorem.)
\begin{lem}
\label{lem:Luzin-for-tests}Let $\{t_{\mu}^{a}\}_{(a,\mu)\in K}$
be a uniform Schnorr integral test. Let $x_{0}\in\mathbb{X}$, $(a_{0},\mu_{0})\in K$
be such that $x_{0}\in\SR{\mu_{0}}{a_{0}}$. Then there is a computable
function $h\colon\mathbb{A}\times\probmeas(\mathbb{X})\times\mathbb{X}\rightarrow[0,\infty)$
such that $h_{\mu}^{a}(x)\leq t_{\mu}^{a}(x)$ and $h_{\mu_{0}}^{a_{0}}(x_{0})=t_{\mu_{0}}^{a_{0}}(x_{0})$.
\end{lem}
\begin{proof}
There is no loss in assuming that $t_{\mu}^{a}\geq1$. (For, if not,
we find $h_{\mu}^{a}$ for $t_{\mu}^{a}+1$ and then $\max\{0,h_{\mu}^{a}-1\}$
is the desired result.) First we prove a useful claim.

\begin{claim*}
Assuming $t_{\mu}^{a}\geq1$, there is a computable function $f\colon[0,\infty)\times K\times\mathbb{X}\rightarrow[0,\infty)$
strictly increasing in the first coordinate such that $t_{\mu}^{a}=\sup_{r}f(r,a,\mu,x)$
and $\int(t_{\mu}^{a}(x)-f(r,a,\mu,x))\,d\mu(x)=2^{-r}$.
\end{claim*}
\begin{proof}[Proof of claim.]
By Lemma~\ref{lem:approx-below}, there is a computable family $\{g_{\mu}^{r;a}\}_{r\in[0,\infty),(a,\mu)\in K}$
of strictly increasing uniformly computable functions such that for
$(a,\mu)\in K$ we have $t_{\mu}^{a}=\sup_{r}g_{\mu}^{r;a}$ and $r,a,\mu\mapsto\int_{{}}g_{\mu}^{r;a}\,d\mu$
is computable.

However, $\int(t_{\mu}^{a}-g_{\mu}^{r;a})\,d\mu$ may not be $2^{-r}.$
Notice that the integral of the difference
\[
F(r,a,\mu)=\int(t_{\mu}^{a}-g_{\mu}^{r;a})\,d\mu=\int_{{}}t_{\mu}^{a}\,d\mu-\int_{{}}g_{\mu}^{r;a}\,d\mu
\]
is computable and in the first coordinate it is a strictly decreasing
function from $[0,\infty)$ into $\mathbb{R}$. Moreover, since $g_{\mu}^{0;a}=0$
and $t_{\mu}^{a}\geq1$, for any $a$ and $\mu$ the greatest value
of $d(r,a,\mu)$ is $d(0,a,\mu)=\int t_{\mu}^{a}\,d\mu\geq1$ and
the lower limit is 
\[
\lim_{r\to\infty}F(r,a,\mu)=\lim_{r\to\infty}\int(t_{\mu}^{a}-g_{\mu}^{r;a})\,d\mu=0
\]
Therefore, we can compute a partial inverse map $F^{-1}(s,a,\mu)$
(for $s\in(0,1]$) such that $F(F^{-1}(s,a,\mu),a,\mu)=s$. We consider
the function $f\colon[0,\infty)\times K\times\mathbb{X}\rightarrow[0,\infty)$
such that 
\[
f(r,a,\mu,x)=g(F^{-1}(2^{-r},a,\mu),a,\mu,x).
\]
Then $t_{\mu}^{a}=\sup\{f(r,a,\mu,x):r\in[0,\infty)\}$ and $\int(t_{\mu}^{a}(x)-f(r,a,\mu,x))\,d\mu(x)=2^{-r}$,
proving the claim.
\end{proof}
As usual, write $f(r,a,\mu,x)$ as $f_{\mu}^{r;a}(x)$. The main idea
of the proof is that the convergence of $f_{\mu}^{n;a}(x)$ to $t_{\mu}^{a}(x)$
is almost uniform. In particular, $t_{\mu}^{a}<f_{\mu}^{2r;a}+2^{-r}$
except with a small probability.

For $m\in\mathbb{N}$, define $h_{\mu}^{m;a}$ as follows. First,
define
\[
g_{\mu}^{m,r;a}(x)=\inf\{f_{\mu}^{2s;a}(x)+2^{-s}:m\leq s\leq r\}.
\]
This is computable in $m,r,a,\mu,x$ ($r\geq m$) since the infimum
over the closed interval $[m,r]$ is computable. Notice $g_{\mu}^{m,r;a}$
is nonincreasing in $r$ and $g_{\mu}^{m,r;a}\leq f_{\mu}^{r;a}+2^{-r}$.
Next, define 
\[
h_{\mu}^{m;a}(x)=\sup\{f_{\mu}^{2r;a}(x):r\in[m,\infty)\ \text{and}\ f_{\mu}^{2r;a}(x)\leq g_{\mu}^{m,r;a}(x)\}.
\]
Clearly, $h_{\mu}^{m;a}\leq t_{\mu}^{a}$ since $t_{\mu}^{a}(x)=\sup_{r}f_{\mu}^{2r;a}(x)$.
By this definition $h_{\mu}^{m;a}(x)$ is clearly lower semicomputable.
However, we also have that $h_{\mu}^{m;a}(x)$ is upper semicomputable
since
\[
h_{\mu}^{m;a}(x)=\inf\{g_{\mu}^{m,r;a}(x):f_{\mu}^{2r;a}(x)\leq g_{\mu}^{m,r;a}(x)\}.
\]
Therefore $h_{\mu}^{m;a}(x)$ is computable in $m,a,\mu,x$.

Now for a fixed $x_{0}\in\mathbb{X}$ and $(a_{0},\mu_{0})\in K$
we have two cases:

In the first case, there exists an $m$ such that for all $r\geq m$,
$f_{\mu_{0}}^{2r;a_{0}}(x_{0})<g_{\mu_{0}}^{m,r;a_{0}}(x_{0})$. In
this case, 
\[
h_{\mu_{0}}^{m;a_{0}}(x_{0})=\sup_{r}f_{\mu_{0}}^{2r;a_{0}}(x_{0})=t_{\mu_{0}}^{a_{0}}(x_{0}).
\]
By letting $h_{\mu_{0}}^{a_{0}}(x_{0})=h_{\mu_{0}}^{m;a_{0}}(x_{0})$
for that choice of $m$ we have the desired conclusion. It remains
to show that the other case cannot hold for $x_{0}\in\SR{\mu_{0}}{a_{0}}$.

In the second case, for all $m$, there exists some $r$ such that
$f_{\mu_{0}}^{2r;a_{0}}(x_{0})=g_{\mu_{0}}^{m,r;a_{0}}(x_{0})$ and
therefore there exists some $s\in[m,r)$ such that $f_{\mu_{0}}^{2r;a_{0}}(x_{0})=g_{\mu_{0}}^{m,r;a_{0}}(x_{0})=f_{\mu_{0}}^{2s;a_{0}}(x_{0})+2^{-s}$.
For a fixed $m$ and the corresponding $r,s$ we have 
\[
t_{\mu_{0}}^{a_{0}}(x_{0})-f_{\mu_{0}}^{2s;a_{0}}(x_{0})\geq f_{\mu_{0}}^{2r;a_{0}}(x_{0})-f_{\mu_{0}}^{2s;a_{0}}(x_{0})=2^{-s}.
\]
In particular, for $n=\lfloor s\rfloor$ we have that $n\geq m$ and
\[
t_{\mu_{0}}^{a_{0}}(x_{0})-f_{\mu_{0}}^{2n;a_{0}}(x_{0})\geq t_{\mu_{0}}^{a_{0}}(x_{0})-f_{\mu_{0}}^{2s;a_{0}}(x_{0})\geq2^{-s}\geq2^{-(n+1)}.
\]
Since $m$ is arbitrary, there exists infinitely many $n$ that 
\[
t_{\mu_{0}}^{a_{0}}(x_{0})-f_{\mu_{0}}^{2n;a_{0}}(x_{0})\geq2^{-(n+1)}.
\]
Now, consider the sum
\[
u_{\mu}^{a}=\sum_{n\in\mathbb{N}}2^{(n+1)}(t_{\mu}^{a}-f_{\mu}^{2n;a}).
\]
Then $u_{\mu}^{a}$ is lower semicomputable and $u_{\mu_{0}}^{a_{0}}(x_{0})=\infty$.
Also by the Monotone Convergence Theorem we can compute the integral
of $u_{\mu}^{a}$ to be 
\[
\int_{{}}u_{\mu}^{a}\,d\mu=\sum_{n\in\mathbb{N}}\left(2^{(n+1)}\int(t_{\mu}^{a}-f_{\mu}^{2n;a})\,d\mu\right)=\sum_{n\in\mathbb{N}}2^{(n+1)}2^{-2n}=2\sum_{n\in\mathbb{N}}2^{-n}=4.
\]
Therefore, $x_{0}\notin\SR{\mu_{0}}{a_{0}}$. This prove the lemma.
\end{proof}
\begin{thm}
\label{thm:vL-sch}Let $\mathbb{X}$ and $\mathbb{Y}$ be computable
metric spaces. Let $\mu\in\probmeas(\mathbb{X})$ and $\nu\in\probmeas(\mathbb{Y})$.
Then
\[
(x,y)\in\SR{\mu\otimes\nu}{}\quad\text{iff}\quad x\in\SR{\mu}{\nu}\text{ and }y\in\SR{\nu}{\mu,x}.
\]
\end{thm}
\begin{proof}[Proof of $\Rightarrow$]
Fix $\mu_{0},\nu_{0},x_{0},y_{0}$. We show the contrapositive. Assume,
$y_{0}\notin\SR{\nu_{0}}{\mu_{0},x_{0}}$. (The other case, $x_{0}\notin\SR{\mu_{0}}{\nu_{0}}$,
is handled similarly.) Then, there is a normalized uniform Schnorr
integral test $\{t_{\nu}^{\mu,x}\}_{\mu\in\probmeas(\mathbb{X}),x\in\mathbb{X},\nu\in\probmeas(\mathbb{Y})}$
such that $t_{\nu_{0}}^{\mu_{0},x_{0}}(y_{0})=\infty$. (Hence $\int_{{}}t_{\nu}^{\mu,x}(y)\,d\nu(y)=1$
for all $\mu,x,\nu$.)

Consider the family of functions $\{s_{\rho}\}_{\rho\in P}$ indexed
by the product measures
\[
P=\{\mu\otimes\nu:\mu,\nu\in\probmeas(\mathbb{X})\},
\]
and where $s_{\mu\otimes\nu}$ is given by
\[
s_{\mu\otimes\nu}(x,y):=t_{\nu}^{\mu,x}(y).
\]

We have that $\{s_{\rho}\}_{\rho\in P}$ is a uniform Schnorr integral
test, since $s_{\mu\otimes\nu}(x,y)$ is lower semicomputable in $\mu\otimes\nu$,
$x$, $y$ (Lemma~\ref{lem:Portmanteau}), and since Fubini's Theorem
gives us 
\[
\iint_{\mathbb{X}\times\mathbb{Y}}s(x,y)\,d(\mu\otimes\nu)(x,y)=\int_{\mathbb{X}}\left(\int_{\mathbb{Y}}t_{\nu}^{\mu,x}(y)\,d\nu(y)\right)d\mu(x)=\int_{\mathbb{X}}1\,d\mu=1.
\]
Since $P$ is effectively closed, and $s_{\mu_{0}\otimes\nu_{0}}(x_{0},y_{0})=\infty$,
by Theorem~\ref{thm:SR-over-K} we have $(x_{0},y_{0})\notin\SR{\mu_{0}\otimes\nu_{0}}{}$.
\end{proof}
\begin{proof}[Proof of $\Leftarrow$]
Fix $x_{0},y_{0},\mu_{0},\nu_{0}$. We show the contrapositive. Assume
$(x_{0},y_{0})\notin\SR{\mu_{0}\otimes\nu_{0}}{}$ and $x_{0}\in\SR{\mu_{0}}{\nu_{0}}$.
Then there is a normalized uniform Schnorr integral test $\{t_{\rho}\}_{\rho\in\probmeas(\mathbb{X}\times\mathbb{Y})}$
such that $t_{\mu_{0}\otimes\nu_{0}}(x_{0},y_{0})=\infty$. (Hence
$\int_{{}}t_{\rho}(x,y)\,d\rho(x,y)=1$ for all $\rho$.)

Consider the family $\{s_{\nu}^{\mu,x}\}_{\mu\in\probmeas(\mathbb{X}),x\in\mathbb{X},\nu\in\probmeas(\mathbb{Y})}$
of lower semicomputable functions given by 
\[
s_{\nu}^{\mu,x}(y):=t_{\mu\otimes\nu}(x,y).
\]
However, $\{s_{\nu}^{\mu,x}\}_{\mu\in\probmeas(\mathbb{X}),x\in\mathbb{X},\nu\in\probmeas(\mathbb{Y})}$
may not be a uniform test since it is possible that $\int_{{}}s_{\nu}^{\mu,x}(y)\,d\nu(y)=\infty$
for some $x$ (albeit a $\mu$-measure zero set of $x$'s for each
$\mu$ and $\nu$). To fix this, consider the following family of
functions 
\[
I_{\mu}^{\nu}(x):=\int_{\mathbb{Y}}s_{\nu}^{\mu,x}(y)\,d\nu(y)=\int_{\mathbb{Y}}t_{\mu\otimes\nu}(x,y)\,d\nu(y).
\]
Then $I_{\mu}^{\nu}(x)$ is lower semicomputable in $\nu$, $\mu$,
$x$ (Lemma~\ref{lem:Portmanteau}), and applying Fubini's Theorem,
\[
\int_{\mathbb{X}}I_{\mu}^{\nu}(x)\,d\mu(x)=\iint_{\mathbb{X}\times\mathbb{Y}}t_{\mu\otimes\nu}(x,y)\,d(\mu\otimes\nu)(x,y)=1.
\]
Since $x_{0}\in\SR{\mu_{0}}{\nu_{0}}$ and $\{I_{\mu}^{\nu}\}_{\nu\in\probmeas(\mathbb{Y}),\mu\in\probmeas(\mathbb{X})}$
is a uniform Schnorr integral test, we have that $\int_{{}}s_{\nu0}^{\mu_{0},x_{0}}(y)\,d\nu_{0}(y)=I_{\mu_{0}}^{\nu_{0}}(x_{0})<\infty$.
However, we need more. 

By Lemma~\ref{lem:Luzin-for-tests}, there is some computable $h_{\mu}^{\nu}\colon\mathbb{X}\rightarrow[0,\infty)$
such that $h\leq I_{\mu}^{\nu}$ and $h_{\mu_{0}}^{\nu_{0}}(x_{0})=I_{\mu_{0}}^{\nu_{0}}(x_{0})$.
Let $\{\hat{s}_{\nu}^{\mu,x}\}$ be the family of lower semicomputable
functions which equals $s_{\nu}^{\mu,x}$ except that we ``cut off''
the enumeration so that $\int_{{}}\hat{s}_{\nu}^{\mu,x}(y)\,d\nu(y)=h_{\mu}^{\nu}(x)$.

Formally, let $\{f_{\nu}^{r;\mu,x}\}$ be the family of continuous
functions from Lemma\ \ref{lem:approx-below} computable in $r,\mu,x,\nu$
($r\in[0,\infty)$) such that $s_{\nu}^{\mu,x}=\sup_{r}f_{\nu}^{r;\mu,x}$.
Then for $\mu,x,\nu$, let 
\[
\hat{s}_{\nu}^{\mu,x}:=\sup\left\{ f_{\nu}^{r;\mu,x}\ :\ r\in[0,\infty)\text{ and }{\textstyle \int_{\mathbb{Y}}f_{\nu}^{r;\mu,x}(y)\,d\nu(y)\leq h_{\mu}^{\nu}(x)}\right\} .
\]
Then $\int_{{}}\hat{s}_{\nu}^{\mu,x}d\nu=h_{\mu}^{\nu}(x)$ as desired.

Then $\{\hat{s}_{\nu}^{\mu,x}\}$ is a uniform Schnorr integral test.
Moreover, since $h_{\mu_{0}}^{\nu_{0}}(x_{0})=I_{\mu_{0}}^{\nu_{0}}(x_{0})$
we have $\hat{s}_{\nu_{0}}^{\mu_{0},x_{0}}=s_{\nu_{0}}^{\mu_{0},x_{0}}$.
Therefore $\hat{s}_{\nu_{0}}^{\mu_{0},x_{0}}(y_{0})=s_{\nu_{0}}^{\mu_{0},x_{0}}(y_{0})=t_{\mu_{0}\otimes\nu_{0}}(x_{0},y_{0})=\infty$.
Hence $y_{0}\notin\SR{\nu_{0}}{\mu_{0},x_{0}}$.
\end{proof}
By symmetry, we instantly have the following corollary.
\begin{cor}
\label{cor:vL_SR-symmetric}Let $\mathbb{X}$ and $\mathbb{Y}$ be
computable metric spaces. Let $\mu\in\probmeas(\mathbb{X})$ and $\nu\in\probmeas(\mathbb{Y})$.
Then
\[
(x\in\SR{\mu}{\nu}\text{ and }y\in\SR{\nu}{\mu,x})\quad\text{iff}\quad(y\in\SR{\nu}{\mu}\text{ and }x\in\SR{\mu}{\nu,y}).
\]
\end{cor}
We can also easily extend Theorem~\ref{thm:vL-sch} and its proof
to include an oracle $a$. Instead we can move the oracle into the
measure by considering the Dirac measure $\delta_{a}$ ($\delta_{a}(\{a\})=1$
and $\delta_{a}(\mathbb{X}\smallsetminus\{a\})=0$).
\begin{prop}
\label{prop:Dirac}$x\in\SR{\mu}a$ if and only if $(x,a)\in\SR{\mu\otimes\delta_{a}}{}$.
\end{prop}
\begin{proof}
Clearly, $a\in\SR{\delta_{a}}{\delta_{a},a}$ since $a$ is an atom
of $\delta_{a}$. Therefore, Van Lambalgen's Theorem (Theorem~\ref{thm:vL-sch})
gives us 
\[
(x,a)\in\SR{\mu\otimes\delta_{a}}{}\quad\text{iff}\quad x\in\SR{\mu}{\delta_{a}}.
\]
It remains to show that $\SR{\mu}{\delta_{a}}=\SR{\mu}a$. Notice
that $\delta_{a}$ is uniformly computable from $a$. To show $a$
is uniformly computable from $\delta_{a}$ (as in Definition~\ref{def:uniform-comp})
consider the set $D\subseteq\finitemeas(\mathbb{X})$ of Dirac measures.
This set $D$ is effectively closed and given a Dirac measure $\delta_{a}$
one can uniformly compute the corresponding atom $a$. By Proposition~\ref{prop:reducible-oracles},
$\SR{\mu}{\delta_{a}}=\SR{\mu}a$ as desired.
\end{proof}

\section{Schnorr randomness for measures on a product space\label{sec:SR-kernels}}

Not all measures on a product space $\mathbb{X}\times\mathbb{Y}$
are product measures of the form $\mu\otimes\nu$. Instead there is
a more general way to construct measures on $\mathbb{X}\times\mathbb{Y}$.
Let $\mu$ be a probability measure on $\mathbb{X}$, and consider
a measurable map $\kappa\colon\mathbb{X}\rightarrow\probmeas(\mathbb{Y})$.
This map $\kappa$ is called a \emph{kernel}. For $x\in\mathbb{X}$
and $B\subseteq\mathbb{Y}$, let $\kappa(B\mid x)$ denote the probability
measure $\kappa(x)$ applied to the measurable set $B$. Similarly,
let $\kappa(\cdot\mid x)$ denote the measure $\kappa(x)$. A measure
$\mu\in\probmeas(\mathbb{X})$ and a kernel $\kappa\colon\mathbb{X}\rightarrow\probmeas(\mathbb{Y})$,
when combined, give a measure $\mu*\kappa$ on $\mathbb{X}\times\mathbb{Y}$
given by
\begin{equation}
(\mu*\kappa)(A\times B)=\int_{A}\kappa(B\mid x)\,d\mu(x)\qquad(A\subseteq\mathbb{X},B\subseteq\mathbb{Y}).\label{eq:kernel-formula}
\end{equation}
Formula (\ref{eq:kernel-formula}) extends to integrals, 
\begin{equation}
\iint_{\mathbb{X}\times\mathbb{Y}}f(x,y)\,(\mu*\kappa)(dxdy)=\int_{\mathbb{X}}\left(\int_{\mathbb{Y}}f(x,y)\,\kappa(dy\mid x)\right)\,\mu(dx).\label{eq:kernel-formula-function}
\end{equation}
Here $(\mu*\kappa)(dxdy)$, $\kappa(dy\mid x)$, and $\mu(dx)$ are,
respectively, alternate notations for $d(\mu*\kappa)(x,y)$, $d\kappa(\cdot\mid x)(y)$,
and $d\mu(x)$. Using Lemmas~\ref{lem:char-comp-meas} and \ref{lem:Portmanteau}
and equation~(\ref{eq:kernel-formula-function}) it is easy to see
that $\mu*\kappa$ is computable uniformly from $\mu$ and $\kappa$. 
\begin{defn}
\label{def:unif-family-kernels}Let $\mathbb{X}$, $\mathbb{Y}$,
and $\mathbb{A}$ be computable metric spaces and let $K\subseteq\mathbb{A}\times\probmeas(\mathbb{X})$
be an effectively closed subset. We will refer to a computable map
$\kappa\colon K\times\mathbb{X}\rightarrow\probmeas(\mathbb{Y})$
as a \emph{uniformly computable family of continuous kernels} and
denote it as $\{\kappa_{\mu}^{a}\colon\mathbb{X}\rightarrow\probmeas(\mathbb{Y})\}_{(a,\mu)\in K}$
where $\kappa_{\mu}^{a}$ and $\kappa_{\mu}^{a}(\cdot\mid x)$, respectively,
denote the kernel $x\mapsto\kappa(a,\mu,x)$ and the measure $\kappa(a,\mu,x)$.

We have this useful extension of Van Lambalgen's Theorem for continuous
kernels.
\end{defn}
\begin{thm}
\label{thm:vL-Sch-kernel}Let $\mathbb{A}$, $\mathbb{X}$, and $\mathbb{Y}$
be computable metric spaces. Let $K\subseteq\mathbb{A}\times\probmeas(\mathbb{X})$
be an effectively closed subset and let $\{\kappa_{\mu}^{a}\colon\mathbb{X}\rightarrow\probmeas(\mathbb{Y})\}_{(a,\mu)\in K}$
be a uniformly computable family of continuous kernels as in Definition~\ref{def:unif-family-kernels}.
For all $(a,\mu)\in K$, $x\in\mathbb{X}$, and $y\in\mathbb{Y}$,
\[
(x,y)\in\SR{\mu*\kappa_{\mu}^{a}}a\quad\text{iff}\quad x\in\SR{\mu}a\text{ and }y\in\SR{\kappa_{\mu}^{a}(\cdot\mid x)}{a,\mu,x}.
\]
\end{thm}
\begin{proof}[Proof of $\Rightarrow$]
Fix $a_{0},\mu_{0},x_{0},y_{0}$. Prove the contrapositive, considering
two cases:

\noindent \emph{Case 1.} Assume $x_{0}\notin\SR{\mu_{0}}{a_{0}}$.
Then, there is a normalized uniform Schnorr integral test $\{t_{\mu}^{a}\}_{a\in\mathbb{A},\mu\in\probmeas(\mathbb{X})}$
such that $t_{\mu_{0}}^{a_{0}}(x_{0})=\infty$. (Hence $\int_{{}}t_{\mu}^{a}(x)\,\mu(dx)=1$
for all $a,\mu$.)

Given any $\rho\in\probmeas(\mathbb{X}\times\mathbb{Y})$, we can
uniformly compute the marginal measure $\rho_{\mathbb{X}}$ given
by $\rho_{\mathbb{X}}(A)=\rho(A\times\mathbb{Y})$. Also, if $\rho=\mu*\kappa$
then $\rho_{\mathbb{X}}=\mu$. Therefore, the following is an effectively
closed subset of $\mathbb{A}\times\probmeas(\mathbb{X}\times\mathbb{Y})$,
\begin{align*}
P & =\{(a,\mu*\kappa_{\mu}^{a}):(a,\mu)\in K\}\\
 & =\{(a,\rho):(a,\rho_{\mathbb{X}})\in K\ \text{and}\ \rho=\rho_{\mathbb{X}}*\kappa_{\rho_{\mathbb{X}}}^{a}\}.
\end{align*}
Now consider, the family of functions $\{s_{\rho}^{a}\}_{(a,\rho)\in P}$
given by
\[
s_{\mu*\kappa_{\mu}^{a}}^{a}(x,y):=t_{\mu}^{a}(x).
\]
To see that $\{s_{\rho}^{a}\}_{(a,\rho)\in P}$ is a uniform Schnorr
integral test, $s_{\mu*\kappa_{\mu}^{a}}^{a}(x,y)$ is lower semicomputable
in $a,$ $\mu*\kappa_{\mu}^{a}$, $x$, $y$ (recall that $\mu$ is
computable in $\mu*\kappa_{\mu}^{a}$). Moreover, using formula (\ref{eq:kernel-formula-function})
and the fact that $s_{\mu*\kappa_{\mu}^{a}}^{a}(x,y)$ does not depend
on $y$, 
\[
\iint_{\mathbb{X}\times\mathbb{Y}}s_{\mu*\kappa_{\mu}^{a}}^{a}(x,y)\,(\mu*\kappa_{\mu}^{a})(dxdy)=\int_{\mathbb{X}}t_{\mu}^{a}(x)\,\left(\int_{\mathbb{Y}}\kappa_{\mu}^{a}(dy\mid x)\right)\,\mu(dx)=\int_{\mathbb{X}}t_{\mu}^{a}(x)\,\mu(dx)=1.
\]
Since $P$ is effectively closed, and since $s_{\mu_{0}*\kappa_{\mu_{0}}^{a_{0}}}^{a_{0}}(x_{0},y_{0})=\infty$,
by Theorem~\ref{thm:SR-over-K} therefore $(x_{0},y_{0})\notin\SR{\mu_{0}*\kappa_{\mu_{0}}^{a_{0}}}{a_{0}}$.

\noindent \emph{Case 2.} Assume $y_{0}\notin\SR{\kappa_{\mu_{0}}^{a_{0}}(\cdot\mid x_{0})}{a_{0},\mu_{0},x_{0}}$.
Then, there is a normalized uniform Schnorr integral test $\{t_{\nu}^{a,\mu,x}\}_{a\in\mathbb{A},\mu\in\probmeas(\mathbb{X}),x\in\mathbb{X},\nu\in\probmeas(\mathbb{Y})}$
such that $t_{\kappa_{\mu_{0}}^{a_{0}}(\cdot\mid x_{0})}^{a_{0},\mu_{0},x_{0}}(y_{0})=\infty$.
(Hence $\int_{{}}t_{\nu}^{a,\mu,x}(y)\,\nu(dy)=1$ for all $a,\mu,x,\nu$.)

Consider the family of functions $\{s_{\rho}^{a}\}_{(a,\rho)\in P}$
(where $P$ is as in Case 1) given by 
\[
s_{\mu*\kappa_{\mu}^{a}}^{a}(x,y):=t_{\kappa_{\mu}^{a}(\cdot\mid x)}^{a,\mu,x}(y).
\]
To see that $\{s_{\rho}^{a}\}_{(a,\rho)\in P}$ is a uniform Schnorr
integral test, $s_{\mu*\kappa_{\mu}^{a}}^{a}(x,y)$ is lower semicomputable
in $a,$ $\mu*\kappa_{\mu}^{a}$, $x$, $y$. Moreover, using formula
(\ref{eq:kernel-formula-function}), 
\[
\iint_{\mathbb{X}\times\mathbb{Y}}s_{\mu*\kappa_{\mu}^{a}}^{a}(x,y)\,(\mu*\kappa_{\mu}^{a})(dxdy)=\int_{\mathbb{X}}\left(\int_{\mathbb{Y}}t_{\kappa_{\mu}^{a}(\cdot\mid x)}^{a,\mu,x}(y)\,\kappa_{\mu}^{a}(dy\mid x)\right)\,\mu(dx)=\int_{\mathbb{X}}1\,\mu(dx)=1.
\]
Since $P$ is effectively closed and $s_{\mu_{0}*\kappa_{\mu_{0}}^{a_{0}}}^{a_{0}}(x_{0},y_{0})=\infty$,
by Theorem~\ref{thm:SR-over-K} we have $(x_{0},y_{0})\notin\SR{\mu_{0}*\kappa_{\mu_{0}}^{a_{0}}}{a_{0}}$.
\end{proof}
\begin{proof}[Proof of $\Leftarrow$]
Fix $a_{0},\mu_{0},x_{0},y_{0}$. Prove the contrapositive. Assume
$(x_{0},y_{0})\notin\SR{\mu_{0}*\kappa_{\mu_{0}}^{a_{0}}}{a_{0}}$
and $x_{0}\in\SR{\mu_{0}}{a_{0}}$. Then there is a normalized uniform
Schnorr integral test $\{t_{\rho}^{a}\}_{a\in\mathbb{A},\rho\in\probmeas(\mathbb{X}\times\mathbb{Y})}$
such that $t_{\mu_{0}*\kappa_{\mu_{0}}^{a_{0}}}^{a_{0}}(x_{0},y_{0})=\infty$.
(Hence $\iint_{{}}t_{\rho}^{a}(x,y)\,\rho(dxdy)=1$ for all $a,\rho$.)

Consider the effectively closed set 
\begin{align*}
P & =\{(a,\mu,x,\kappa_{\mu}^{a}(\cdot\mid x)):a\in\mathbb{A},\mu\in\probmeas(\mathbb{X}),x\in\mathbb{X}\}\\
 & \subseteq\mathbb{A}\times\probmeas(\mathbb{X})\times\mathbb{X}\times\probmeas(\mathbb{Y}),
\end{align*}
and consider the family of functions $\{s_{\nu}^{a,\mu,x}\}_{(a,\mu,x,\nu)\in P}$
given by
\[
s_{\kappa_{\mu}^{a}(\cdot\mid x)}^{a,\mu,x}(y):=t_{\mu*\kappa_{\mu}^{a}}^{a}(x,y).
\]
However, $\{s_{\nu}^{a,\mu,x}\}_{(a,\mu,x,\nu)\in P}$ may not be
a uniform test since it is possible that $\int_{{}}s_{\kappa_{\mu}^{a}(\cdot\mid x)}^{a,\mu,x}(y)\,\kappa_{\mu}^{a}(dy\mid x)=\infty$
for some $x$. To fix this, consider the following family of functions
\[
I_{\mu}^{a}(x):=\int_{\mathbb{Y}}s_{\kappa_{\mu}^{a}(\cdot\mid x)}^{a,\mu,x,}(y)\,\kappa_{\mu}^{a}(dy\mid x)=\int_{\mathbb{Y}}t_{\mu*\kappa_{\mu}^{a}}^{a}(x,y)\,\kappa_{\mu}^{a}(dy\mid x).
\]
Then $I_{\mu}^{a}(x)$ is lower semicomputable in $a$, $\mu$, $x$
(Lemma~\ref{lem:Portmanteau}), and using formula\ (\ref{eq:kernel-formula-function}),
\[
\int_{\mathbb{X}}I_{\mu}^{a}(x)\,\mu(dx)=\int_{\mathbb{X}}\left(\int_{\mathbb{Y}}t_{\mu*\kappa_{\mu}^{a}}^{a}(x,y)\,\kappa_{\mu}^{a}(dy\mid x)\right)\mu(dx)=\iint_{\mathbb{X\times\mathbb{Y}}}t_{\mu*\kappa_{\mu}^{a}}^{a}(x,y)\,(\mu*\kappa_{\mu}^{a})(dxdy)=1.
\]
Since $x_{0}\in\SR{\mu_{0}}{a_{0}}$, we have that $\int_{{}}s_{\kappa_{\mu_{0}}^{a_{0}}(\cdot\mid x_{0})}^{a_{0},\mu_{0},x_{0}}(y)\,\kappa_{\mu_{0}}^{a_{0}}(dy\mid x_{0})=I_{\mu_{0}}^{a_{0}}(x_{0})<\infty$.
However, we need more. 

By Lemma~\ref{lem:Luzin-for-tests}, there is some computable $h_{\mu}^{a}\colon\mathbb{X}\rightarrow[0,\infty)$
such that $h_{\mu}^{a}\leq I_{\mu}^{a}$ and $h_{\mu_{0}}^{a_{0}}(x_{0})=I_{\mu_{0}}^{a_{0}}(x_{0})$.
Let $\{\hat{s}_{\nu}^{a,\mu,x}\}_{(a,\mu,x,\nu)\in P}$ be the family
of lower semicomputable functions which equals $s_{\nu}^{a,\mu,x}(y)$
except that we ``cut off'' the enumeration so that $\int_{{}}\hat{s}_{\nu}^{a,\mu,x}(y)\,\nu(dy)=h_{\mu}^{a}(x)$.

Formally, let $\{f_{\nu}^{r;a,\mu,x}\}$ be the family of continuous
functions from Lemma\ \ref{lem:approx-below} computable in $r,a,\mu,x,\nu$
($r\in[0,\infty)$) such that $s_{\nu}^{a,\mu,x}=\sup_{r}f_{\nu}^{r;a,\mu,x}$.
Then for $a,\mu,x,\nu$, let 
\[
\hat{s}_{\nu}^{a,\mu,x}:=\sup\left\{ f_{\nu}^{r;a,\mu,x}\ :\ r\in[0,\infty)\text{ and }{\textstyle \int_{\mathbb{Y}}f_{\nu}^{a,\mu,x}(y)\,\nu(dy)\leq h_{\mu}^{a}(x)}\right\} .
\]
Then $\int_{{}}\hat{s}_{\nu}^{a,\mu,x}(y)\,\nu(dy)=h_{\mu}^{a}(x)$
as desired.

Then $\{\hat{s}_{\nu}^{a,\mu,x}\}_{(a,\mu,x,\nu)\in P}$ is a uniform
Schnorr integral test and $\hat{s}_{\kappa_{\mu_{0}}^{a_{0}}(\cdot\mid x_{0})}^{a_{0},\mu_{0},x_{0}}=s_{\kappa_{\mu_{0}}^{a_{0}}(\cdot\mid x_{0})}^{a_{0},\mu_{0},x_{0}}$.
Therefore $\hat{s}_{\kappa_{\mu_{0}}^{a_{0}}(\cdot\mid x_{0})}^{a_{0},\mu_{0},x_{0}}(y_{0})=s_{\kappa_{\mu_{0}}^{a_{0}}(\cdot\mid x_{0})}^{a_{0},\mu_{0},x_{0}}(y_{0})=t_{\mu_{0}*\kappa_{\mu_{0}}^{a_{0}}}^{a_{0}}(x_{0},y_{0})=\infty$.
Hence $y_{0}\notin\SR{\kappa_{\mu_{0}}^{a_{0}}(\cdot\mid x_{0})}{a_{0},\mu_{0},x_{0}}$.
\end{proof}
\begin{rem}
Consider the case when $\cont(\mathbb{X},\probmeas(\mathbb{Y}))$
is a computable metric space, for example when $\mathbb{X}=\mathbb{Y}=\{0,1\}^{\mathbb{N}}$.\footnote{Specially, we want $\cont(\mathbb{X},\probmeas(\mathbb{Y}))$ to have
a computable metric space structure such that the evaluation map $\eval\colon\cont(\mathbb{X},\probmeas(\mathbb{Y}))\times\mathbb{X}\to\probmeas(\mathbb{Y})$
given by $\eval(\kappa,x)=\kappa(x)$ is computable. We remark that
this is true when $\mathbb{X}$ is effectively locally compact (see,
e.g., \cite[Def.~7.15]{Bienvenu.Gacs.Hoyrup.ea:2011}). (In contrast,
if $\mathbb{X}$ is not locally compact, then $\cont(\mathbb{X},\probmeas(\mathbb{Y}))$
is likely not a Polish space.)}  In this case, we can use the kernel $\kappa$ directly as an oracle.
Theorem~\ref{thm:vL-Sch-kernel} becomes
\[
(x,y)\in\SR{\mu*\kappa}{\kappa}\quad\text{iff}\quad x\in\SR{\mu}{\kappa}\text{ and }y\in\SR{\kappa(\cdot\mid x)}{\mu,\kappa,x}.
\]
\end{rem}
\begin{rem}
\label{rem:meas-kernels}We have only considered the case of continuous
kernels $\kappa$. In general, every probability measure $\rho$ on
$\mathbb{X}\times\mathbb{Y}$ can be decomposed into a probability
measure $\rho_{\mathbb{X}}$ on $\mathbb{X}$ (called the \emph{marginal probability measure})
and a (not necessarily continuous) kernel $\kappa\colon\mathbb{X}\rightarrow\probmeas(\mathbb{Y})$
(called the \emph{conditional probability}). The conditional probability
is sometimes denoted $\rho(\cdot\mid\cdot)$ where $\rho(B\mid x)=\kappa(B\mid x)$.

Even when $\rho$ is computable, the conditional probability $\rho(\cdot\mid\cdot)$
may not be computable in any nice sense. (See Bauwens~\cite[Cor.~3]{Bauwens:}
and Bauwens, Shen, and Takahashi \cite[Ex.~3]{Bauwens.Shen.Takahashi:2016},
also Ackerman, Freer, and Roy \cite{Ackerman.Freer.Roy:2011}.)

In general, one would like to show that for any probability measure
$\rho$ on $\mathbb{X}\times\mathbb{Y}$ that the following holds,
\[
(x,y)\in\SR{\rho}{\rho(\cdot\mid\cdot)}\quad\text{iff}\quad x\in\SR{\rho_{\mathbb{X}}}{\rho(\cdot\mid\cdot)}\text{ and }y\in\SR{\kappa(\cdot\mid x)}{\rho_{\mathbb{X}},\rho(\cdot\mid\cdot),x}.
\]
While this result should be true, to formalize and prove it would
take us too far afield. (Notice it is not even obvious in which computable
metric space $\mathbb{A}$ the conditional probability $\rho(\cdot\mid\cdot)$
lives.)
\end{rem}
\begin{rem}
Theorem~\ref{thm:vL-Sch-kernel} also holds for Martin-Löf randomness.
See, for example, Takahashi \cite{Takahashi:2008,Takahashi:2011a},
Bauwens \cite[Thm.~4]{Bauwens:}, and the survey by Bauwens, Shen,
and Takahashi \cite[Thm.~5]{Bauwens.Shen.Takahashi:2016}.
\end{rem}

\section{\label{sec:SR-maps}Schnorr randomness and measure-preserving maps}

In this section we explore how Schnorr randomness behaves along measure-preserving
maps. The main theorem is another version of Van Lambalgen's Theorem.
In this section, let $\mathbb{X}$ and $\mathbb{Y}$ be computable
metric spaces, and let $\mu$ be a probability measure on $\mathbb{X}$.
Let $T\colon\mathbb{X}\rightarrow\mathbb{Y}$ be a measurable map.
Then the \emph{pushforward measure} $\mu_{T}$ is the probability
measure given by $\mu_{T}(B)=\mu(T^{-1}(A))$. This formula extends
to the following change of basis formula,
\begin{equation}
\int_{\mathbb{Y}}f(y)\,\mu_{T}(dy)=\int_{\mathbb{X}}f(T(x))\,\mu(dx).\label{eq:change-of-basis}
\end{equation}

\begin{defn}
\label{def:unif-family-maps}Let $\mathbb{X}$, $\mathbb{Y}$, and
$\mathbb{A}$ be computable metric spaces and let $K\subseteq\mathbb{A}\times\probmeas(\mathbb{X})$
be a $\Pi_{1}^{0}$ subset. We will refer to a computable map $T\colon K\times\mathbb{X}\rightarrow\mathbb{Y}$
as a \emph{uniformly computable family of continuous maps} and denote
it as $\{T_{\mu}^{a}\colon\mathbb{X}\rightarrow\mathbb{Y}\}_{(a,\mu)\in K}$
where $T_{\mu}^{a}$ denotes the map $x\mapsto T(a,\mu,x)$.
\end{defn}
Note that the map $a,\mu\mapsto\mu_{T_{\mu}^{a}}$ is computable (using
Lemmas~\ref{lem:char-comp-meas} and \ref{lem:Portmanteau} and with
the change of basis formula (\ref{eq:change-of-basis})).

This next proposition concerns randomness conservation which is an
important topic in its own right \cite{Hoyrup.Rojas:2009b,Bienvenu.Porter:2012,Rute:c}.
If $\mu$ is a computable measure, $T\colon\mathbb{X}\to\mathbb{Y}$
is a computable map, and $x\in\SR{\mu}{}$, then $T(x)\in\SR{\mu_{T}}{}$
\cite[Thm.~4.1]{Bienvenu.Porter:2012}. This next proposition generalizes
this fact to a family of continuous maps.
\begin{prop}[Randomness conservation]
\label{prop:rand-preserve-cont}Let $\mathbb{A}$, $\mathbb{X}$,
and $\mathbb{Y}$ be computable metric spaces. Let $K\subseteq\mathbb{A}\times\probmeas(\mathbb{X})$
be an effectively closed subset and let $\{T_{\mu}^{a}\colon\mathbb{X}\rightarrow\mathbb{Y}\}_{(a,\mu)\in K}$
be a uniformly computable family of continuous maps as in Definition~\ref{def:unif-family-maps}.
Then for all $(a,\mu)\in K$ and $x\in\mathbb{X}$, 
\[
x\in\SR{\mu}a\quad\text{implies}\quad T_{\mu}^{a}(x)\in\SR{\mu_{T_{\mu}^{a}}}{a,\mu}.
\]
\end{prop}
\begin{proof}
Fix $a_{0},\mu_{0},x_{0}$. Let $\nu_{0}$ denote the pushforward
$(\mu_{0})_{T_{\mu_{0}}^{a_{0}}}$ and let $y_{0}=T_{\mu}^{a}(x_{0})$.
Assume $y_{0}\notin\SR{\nu_{0}}{a_{0},\mu_{0}}$. Then there is a
normalized uniform Schnorr integral test $\{t_{\nu}^{a,\mu}\}_{a\in\mathbb{A},\mu\in\probmeas(\mathbb{X}),\nu\in\probmeas(\mathbb{Y})}$
such that $t_{\nu_{0}}^{a_{0},\mu_{0}}(y_{0})=\infty$. Consider the
uniform Schnorr integral test $\{s_{\mu}^{a}\}_{a\in\mathbb{A},\mu\in\probmeas(\mathbb{X})}$
given by $s_{\mu}^{a}(x)=t_{\mu_{T_{\mu}^{a}}}^{a,\mu}(T_{\mu}^{a}(x))$.
By the change of basis formula (\ref{eq:change-of-basis}), 
\[
\int_{\mathbb{X}}s_{\mu}^{a}(x)\,\mu(dx)=\int_{\mathbb{Y}}t_{\mu_{T_{\mu}^{a}}}^{a,\mu}(y)\,\mu_{T_{\mu}^{a}}(dy)=1.
\]
Then $s_{\mu_{0}}^{a_{0}}(x_{0})=\infty$, so $x_{0}\notin\SR{\mu_{0}}{a_{0}}$.
\end{proof}
This next proposition concerns isomorphic measure spaces. A important
application of randomness preservation is that computably isomorphic
measure spaces $(\mathbb{X},\mu)$ and $(\mathbb{Y},\nu)$ have ``identical''
Schnorr random points. Specifically, if $\mu$ and $\nu$ are computable
and there is a pair of computable maps $T\colon\mathbb{X}\to\mathbb{Y}$
and $S\colon\mathbb{Y}\to\mathbb{X}$ which are measure-preserving
and inverses of each other, then $x\in\SR{\mu}{}$ if and only if
$T(x)\in\SR{\nu}{}$. This next proposition generalizes this fact
to a family of continuous isomorphisms.
\begin{prop}
\label{prop:iso-preserve-cont}Let $\mathbb{A}$, $\mathbb{X}$, and
$\mathbb{Y}$ be computable metrics spaces. Let $K\subseteq\mathbb{A}\times\probmeas(\mathbb{X})$
and $L\subseteq\mathbb{A}\times\probmeas(\mathbb{Y})$ be effectively
closed subsets and let $\{T_{\mu}^{a}\colon\mathbb{X}\rightarrow\mathbb{Y}\}_{(a,\mu)\in K}$
and $\{S_{\nu}^{a}\colon\mathbb{Y}\rightarrow\mathbb{X}\}_{(a,\nu)\in L}$
be uniformly computable families of continuous maps as in Definition~\ref{def:unif-family-maps}.
Further, assume these two families of maps are inverses in the following
sense:

\begin{itemize}
\item If $(a,\mu)\in K$ and $\nu=\mu_{T_{\mu}^{a}}$, then $(a,\nu)\in L$,
$\mu=\nu_{S_{\nu}^{a}}$, and 
\[
S_{\nu}^{a}\circ T_{\mu}^{a}=\id_{\mathbb{X}}\quad(\mu\text{-a.s.}).
\]
\item If $(a,\nu)\in L$, and $\mu=\nu_{S_{\nu}^{a}}$, then $(a,\mu)\in K$,
$\nu=\mu_{T_{\mu}^{a}}$, and 
\[
T_{\mu}^{a}\circ S_{\nu}^{a}=\id_{\mathbb{Y}}\quad(\nu\text{-a.s.}).
\]
\end{itemize}
\noindent For every $(a,\mu)\in K$ with $\nu=\mu_{T_{\mu}^{a}}$,
and every $x\in\mathbb{X}$ and $y\in\mathbb{Y}$, the following are
equivalent:

\begin{enumerate}
\item $x\in\SR{\mu}a$ and $y=T_{\mu}^{a}(x)$.
\item $y\in\SR{\nu}a$ and $x=S_{\nu}^{a}(y)$.
\end{enumerate}
\end{prop}
\begin{proof}
Fix $x_{0},y_{0},\mu_{0},\nu_{0},a_{0}$ where $(a_{0},\mu_{0})\in K$,
$\nu_{0}=(\mu_{0})_{T_{\mu_{0}}^{a_{0}}}$. By symmetry it is enough
to show (1) implies (2). Assume $x_{0}\in\SR{\mu_{0}}{a_{0}}$. Then
by Proposition~\ref{prop:rand-preserve-cont}, we have $T_{\mu_{0}}^{a_{0}}(x_{0})\in\SR{\nu_{0}}{a_{0}}$.

It remains to show that $x_{0}=S_{\nu_{0}}^{a_{0}}(T_{\mu_{0}}^{a_{0}}(x_{0}))$.
Notice that the set 
\[
\{x\in\mathbb{X}:x=S_{\nu_{0}}^{a_{0}}(T_{\mu_{0}}^{a_{0}}(x))\}
\]
is a closed set of $\mu_{0}$-measure one, so this set contains the
support of $\mu_{0}$, and by Proposition~\ref{prop:support} also
contains $x_{0}$.
\end{proof}
Now, we show how one can modify Van Lambalgen's Theorem for kernels
(Theorem~\ref{thm:vL-Sch-kernel}) to give a useful result about
randomness along continuous maps.

For each measurable map $T\colon\mathbb{X}\rightarrow\mathbb{Y}$,
there is a corresponding map $y\mapsto\mu(\cdot\mid T=y)$, referred
to as the \emph{conditional probability} of $T$, which is a measurable
map of type $\mathbb{Y}\rightarrow\probmeas(\mathbb{X})$ (so it is
a kernel). It is the unique such map (up to $\nu$-a.s.\ equivalence)
satisfying the property 
\[
\mu(A\cap T^{-1}(B))=\int_{B}\mu(A\mid T=y)\,\mu_{T}(dy)
\]
for measurable sets $A\subseteq\mathbb{X}$ and $B\subseteq\mathbb{Y}$.

Consider the space $\mathbb{Y}\times\mathbb{X}$ and the measurable
map $(T,\id_{\mathbb{X}})\colon\mathbb{X}\rightarrow\mathbb{Y}\times\mathbb{X}$
given by $x\mapsto(T(x),x)$. Let $\mu_{(T,\id_{\mathbb{X}})}$ denote
the pushforward of $\mu$ along $(T,\id_{\mathbb{X}})$. (So $\mu_{(T,\id_{\mathbb{X}})}$
is a probability measure on $\mathbb{Y}\times\mathbb{X}$ which is
supported on the inverted graph of $T$.) Now for $B\subseteq\mathbb{Y}$
and $A\subseteq\mathbb{X}$ we have that 
\[
(\mu_{(T,\id_{\mathbb{X}})})(B\times A)=\mu(A\cap T^{-1}(B))=\int_{B}\mu(A\mid T=y)\,\mu_{T}(dy).
\]
Comparing this to equation~(\ref{eq:kernel-formula}) we have that
\[
\mu_{(T,\id_{\mathbb{X}})}=\mu_{T}*\mu(\cdot\mid T=\cdot)
\]
where $\mu(\cdot\mid T=\cdot)$ denotes the kernel $y\mapsto\mu(\cdot\mid T=y)$.
Now, we can apply our Van Lambalgen's Theorem for kernels (Theorem~\ref{thm:vL-Sch-kernel})
to get the following version of Van Lambalgen's Theorem for maps (where
the conditional probability is continuous.)
\begin{thm}
\label{thm:vL-mp-maps}Let $\mathbb{A}$, $\mathbb{X}$ and $\mathbb{Y}$
be computable metric spaces. Let $K\subseteq\mathbb{A}\times\probmeas(\mathbb{X})$
be an effectively closed subset and let $\{T_{\mu}^{a}\colon\mathbb{X}\rightarrow\mathbb{Y}\}_{(a,\mu)\in K}$
be a uniformly computable family of continuous maps as in Definition~\ref{def:unif-family-maps}.
Moreover, assume that the conditional probability map $a,\mu,y\mapsto\mu(\cdot\mid T_{\mu}^{a}=y)$
is a computable map of type $K\times\mathbb{Y}\to\probmeas(\mathbb{X})$.
For $(a,\mu)\in K$, $x\in\mathbb{X}$ and $y\in\mathbb{Y}$ the following
are equivalent:

\begin{enumerate}
\item $x\in\SR{\mu}a$ and $y=T_{\mu}^{a}(x)$.
\item $y\in\SR{\mu_{T_{\mu}^{a}}}{a,\mu}$ and $x\in\SR{\mu(\cdot\mid T_{\mu}^{a}=y)}{a,\mu,y}.$
\end{enumerate}
\end{thm}
\begin{proof}
Consider the space $\mathbb{Y}\times\mathbb{X}$ and let $(T_{\mu}^{a},\id_{\mathbb{X}})\colon\mathbb{X}\rightarrow\mathbb{Y}\times\mathbb{X}$
denote the continuous map $x\mapsto(T_{\mu}^{a}(x),x)$. Let $\mu_{(T_{\mu}^{a},\id_{\mathbb{X}})}$
denote the pushforward of $\mu$ along $(T_{\mu}^{a},\id_{\mathbb{X}})$.
This measure $\mu_{(T_{\mu}^{a},\id_{\mathbb{X}})}$ is supported
on the graph $\{(y,x):y=T_{\mu}^{a}(x)\}$ of the map $T_{\mu}^{a}$,
and $(T_{\mu}^{a},\id_{\mathbb{X}})$ is an isomorphism from $\mu$
onto $\mu_{(T_{\mu}^{a},\id_{\mathbb{X}})}$. Therefore, by Proposition~\ref{prop:iso-preserve-cont}\footnote{To apply Proposition~\ref{prop:iso-preserve-cont}, notice that the
computable projection map $\pi_{\mathbb{X}}\colon\mathbb{Y}\times\mathbb{X}\to\mathbb{X}$
given by $(y,x)\mapsto x$ is the inverse of $(T_{\mu}^{a},\id_{\mathbb{X}})$
in that $\pi_{\mathbb{X}}\circ(T_{\mu}^{a},\id_{\mathbb{X}})=\id_{\mathbb{X}}$
($\mu$-a.s.) and $(T_{\mu}^{a},\id_{\mathbb{X}})\circ\pi_{\mathbb{X}}=\id_{\mathbb{Y}\times\mathbb{X}}$
($\mu_{(T_{\mu}^{a},\id_{\mathbb{X}})}$-a.s.). Now consider the following
table where the symbols to the left of $\rightsquigarrow$ are those
used in Proposition~\ref{prop:iso-preserve-cont} and the ones to
the right are the corresponding values in this proof:
\begin{align*}
\mathbb{X} & \rightsquigarrow\mathbb{X} & \mu & \rightsquigarrow\mu & T_{\mu}^{a}\colon\mathbb{X}\rightarrow\mathbb{Y} & \rightsquigarrow(T_{\mu}^{a},\id_{\mathbb{X}})\colon\mathbb{X}\rightarrow\mathbb{Y}\times\mathbb{X}\\
\mathbb{Y} & \rightsquigarrow\mathbb{Y}\times\mathbb{X} & \nu & \rightsquigarrow\mu_{(T_{\mu}^{a},\id_{\mathbb{X}})} & S_{\nu}^{a}\colon\mathbb{Y}\rightarrow\mathbb{X} & \rightsquigarrow\pi_{\mathbb{X}}\colon\mathbb{Y}\times\mathbb{X}\to\mathbb{X}
\end{align*}
Also, use $(a,\mu)$ as the oracle instead of just $a$. Then we have
the following where the second equivalence follows from Proposition~\ref{prop:iso-preserve-cont}:
\begin{align*}
(x\in\SR{\mu}{a,\mu}\ \text{and}\ y=T_{\mu}^{a}(x)) & \text{\quad iff\quad}(x\in\SR{\mu}{a,\mu}\text{ and }(y,x)=(T_{\mu}^{a},\id_{\mathbb{X}})(x))\\
 & \text{\quad iff\quad}((y,x)\in\SR{\mu_{(T_{\mu}^{a},\id_{\mathbb{X}})}}{a,\mu}\text{ and }x=\pi_{\mathbb{X}}(y,x))\\
 & \text{\quad iff\quad}(y,x)\in\SR{\mu_{(T_{\mu}^{a},\id_{\mathbb{X}})}}{a,\mu}
\end{align*}
}
\[
\left(x\in\SR{\mu}{a,\mu}\ \text{and}\ y=T_{\mu}^{a}(x)\right)\quad\text{iff}\quad(y,x)\in\SR{\mu_{(T_{\mu}^{a},\id_{\mathbb{X}})}}{a,\mu}.
\]
Now, by the discussion above, $\mu_{(T_{\mu}^{a},\id_{\mathbb{X}})}=\mu_{T_{\mu}^{a}}*\mu(\cdot\mid T_{\mu}^{a}=\cdot)$.
By Theorem~\ref{thm:vL-Sch-kernel} (with $x$ and $y$ switched),
we have
\[
(y,x)\in\SR{\mu_{(T_{\mu}^{a},\id_{\mathbb{X}})}}{a,\mu}\quad\text{iff}\quad\left(y\in\SR{\mu_{T_{\mu}^{a}}}{a,\mu}\ \text{and}\ x\in\SR{\mu(\cdot\mid T_{\mu}^{a}=y)}{a,\mu,\mu_{T_{\mu}^{a}},y}\right).
\]
Combining both results gives us the desired result, except that we
want to remove the oracle $\mu_{T_{\mu}^{a}}$ from $\SR{\mu(\cdot\mid T_{\mu}^{a}=y)}{a,\mu,\mu_{T_{\mu}^{a}},y}$.
This is handled by Proposition~\ref{prop:reducible-oracles}, since
the pair $(a,\mu)$ uniformly computes $\mu_{T_{\mu}^{a}}$.
\end{proof}
We instantly get the following corollary.
\begin{cor}
\label{cor:rand-pres-nrfn}Let $\mathbb{A}$, $\mathbb{X}$ and $\mathbb{Y}$
be computable metric spaces. Let $K\subseteq\mathbb{A}\times\probmeas(\mathbb{X})$
be effectively closed and let $\{T_{\mu}^{a}\colon\mathbb{X}\rightarrow\mathbb{Y}\}_{(a,\mu)\in K}$
be a uniformly computable family of continuous maps as in Definition~\ref{def:unif-family-maps}.
Moreover, assume that the conditional probability map $a,\mu,y\mapsto\mu(\cdot\mid T_{\mu}^{a}=y)$
is a computable map of type $K\times\mathbb{Y}\to\probmeas(\mathbb{X})$.
For $(a,\mu)\in K$, $x\in\mathbb{X}$ and $y\in\mathbb{Y}$ the following
both hold:

\begin{enumerate}
\item (Randomness conservation) If $x\in\SR{\mu}a$ then $T_{\mu}^{a}(x)\in\SR{\mu_{T_{\mu}^{a}}}a$.
\item (No-randomness-from-nothing) If $y\in\SR{\mu_{T}}{a,\mu}$, then there
exists some $x\in\SR{\mu}a$ such that $T_{\mu}^{a}(x)=y$.
\end{enumerate}
\end{cor}
The first item is a weaker version of the randomness conservation
result given in Proposition~\ref{prop:rand-preserve-cont}. In Proposition~\ref{prop:rand-preserve-cont}
we did not require that the conditional probability be computable
in any sense. 

However, for the the second item, also known as \emph{no randomness ex nihilo},
the condition that $a,\mu,y\mapsto\mu(\cdot\mid T_{\mu}^{a}=y)$ is
computable is basically necessary. Indeed, Rute~\cite[Thm.~25]{Rute:c}
constructed a computable $\lambda$-measure-preserving map $T\colon\{0,1\}^{\mathbb{N}}\rightarrow\{0,1\}^{\mathbb{N}}$
for which no-randomness-from-nothing fails for Schnorr randomness.\footnote{As a technical point, Rute's map is almost-everywhere computable (so
the map is partial), but it can be modified to be a total computable
map $T\colon\mathbb{X}\rightarrow\{0,1\}^{\mathbb{N}}$ by letting
$\mathbb{X}=\dom T$. (One must use a computable version of Alexandrov's
Theorem\textemdash that every ${\bf \Pi}_{2}^{0}$ subspace of a Polish
space is Polish\textemdash to show that $\mathbb{X}$ is a computable
metric space, and one must show that $\lambda\upharpoonright\mathbb{X}$
is still a computable measure on $\mathbb{X}$ with the same Schnorr
randoms as $\lambda$ on $\{0,1\}^{\mathbb{N}}$.)}

Here is an application of Corollary~\ref{cor:rand-pres-nrfn}.
\begin{example}[Schnorr random Brownian motion]
Let $\mathbb{P}$ denote the Wiener measure on $\cont([0,1])$, that
is the measure of Brownian motion. Consider the map $T\colon\omega\mapsto\omega(1)$
which sends each Brownian motion path $\omega\in\cont([0,1])$ to
its value at $1$. Randomness conservation (Proposition~\ref{prop:rand-preserve-cont})
tells us that if $\omega$ is a Schnorr random Brownian motion, then
$\omega(1)$ is Schnorr random for the Gaussian distribution on $\mathbb{R}$.
(This is easily seen to be equivalent to $\omega(1)$ being Schnorr
random for the Lebesgue measure on $\mathbb{R}$.)

However, now we have the tools to prove the converse direction (which
remained open until now). Assume $a\in\mathbb{R}$ is Schnorr random
for the Gaussian distribution. We wish to find some Schnorr random
Brownian motion $\omega$ such that $\omega(1)=a$. The conditional
probability $\mathbb{P}(\cdot\mid T=a)$ is the probability distribution
associated with being a Brownian motion path $\omega$ which satisfies
$\omega(1)=a$. Such an object is called a \emph{Brownian bridge ending at level $a$}.
Moreover the probability measure $\mathbb{P}(\cdot\mid T=a)$ is uniformly
computable in $a$ since it is just the pushforward of $\mathbb{P}$
along the computable map
\[
\omega\mapsto(\omega(t)-\omega(1)t+at)_{0\leq t\leq1}
\]
which transforms any Brownian motion path $\omega$ into a Brownian
bridge ending at level $a$ (see equation (4) in Pitman \cite{Pitman:1999}).
Therefore, by Corollary~\ref{cor:rand-pres-nrfn}, there is some
Schnorr random Brownian motion $\omega$ such that $\omega(1)=a$.
More generally, Theorem~\ref{thm:vL-mp-maps} tells us a Schnorr
random Brownian motion $\omega$ is exactly composed of a Schnorr
random $a$ (for $\omega(1)$) and a Schnorr random Brownian bridge
ending at level $a$ (for the rest of $\omega$).
\end{example}
\begin{rem}
\label{rem:meas-maps}The results in this section are not as practical
as they could be because we assume $T$ and $y\mapsto\mu(\cdot\mid T=y)$
are both continuous. There are many natural examples where one or
both of these maps is discontinuous. 

For example, the map $T\colon2^{\mathbb{N}}\to[0,1]$ which maps a
binary sequence $x\in2^{\mathbb{N}}$ to its maximum initial frequency
of $1$s, $T(x)=\max_{n}\frac{1}{n}\sum_{k<n}x_{k}$, is not continuous.
Also, the ``time inversion of Brownian motion'' $\omega(t)\mapsto\frac{1}{t}\omega(1/t)$
is a well-known isomorphism from the Wiener measure onto itself, but
it is not continuous as a map of type $C([0,\infty))\to C([0,\infty))$.
(Specifically, if $\xi(t)=t\omega(1/t)$ then $\xi(0)=\lim_{t\to0}t\omega(1/t)=0$
almost surely, but the modulus of continuity of $\xi$ near $0$ is
not continuous in $\omega$. It depends on the rate of convergence
of $\lim_{t\to0}t\omega(1/t)$.) Last, consider the projection map
$S\colon\{0,1\}\times2^{\mathbb{N}}\times2^{\mathbb{N}}\to2^{\mathbb{N}}$
which maps $(n,x_{0},x_{1})\mapsto x_{n}$. This map is clearly continuous,
and even computable. However, let $\mu$ be the measure $\mu_{1}\otimes\mu_{2}\otimes\mu_{3}$
where $\mu_{1}$ is a uniform measure on $\{0,1\},$ $\mu_{2}$ is
the fair coin measure, and $\mu_{3}$ is a Bernoulli measure with
weight $1/3$. We claim the conditional probability map $y\mapsto\mu(\cdot\mid S=y)$
is not continuous as follows. Given $y\in2^{\mathbb{N}}$, the conditional
probability measure $\mu(\cdot\mid S=y)$ would almost surely have
to be concentrated on $\{0\}\times2^{\mathbb{N}}\times2^{\mathbb{N}}$
if $\lim_{n}\frac{1}{n}\sum_{k<n}y_{k}=1/2$ and on $\{1\}\times2^{\mathbb{N}}\times2^{\mathbb{N}}$
if $\lim_{n}\frac{1}{n}\sum_{k<n}y_{k}=1/3$. However, $y\mapsto\lim_{n}\frac{1}{n}\sum_{k<n}y_{k}$
is not a continuous map.

Nonetheless, it should be possible to generalize the results in this
section to measurable maps and measurable conditional probabilities
(including the examples just mentioned). However, just as with Remark~\ref{rem:meas-kernels},
this project is beyond the scope of this paper.
\end{rem}
\begin{rem}
If $\cont(\mathbb{X},\mathbb{Y})$ and $\cont(\mathbb{Y},\probmeas(\mathbb{X}))$
are both computable metric spaces (which usually requires that $\mathbb{X}$
and $\mathbb{Y}$ be effectively locally compact), then we can use
continuous maps $T$ and $\mu(\cdot\mid T=\cdot)$ as oracles avoiding
the need for uniform families of continuous maps. For example, Theorem~\ref{thm:vL-mp-maps}
says the following are equivalent:

\end{rem}
\begin{enumerate}
\item $x\in\SR{\mu}{T,\mu(\cdot\mid T=\cdot)}$ and $y=T(x)$.
\item $y\in\SR{\mu_{T}}{\mu,T,\mu(\cdot\mid T=\cdot)}$ and $x\in\SR{\mu(\cdot\mid T=y)}{\mu,T,\mu(\cdot\mid T=\cdot),y}.$
\end{enumerate}
\begin{rem}
All the theorems in this section hold for Martin-Löf randomness as
well. To our knowledge Theorem~\ref{thm:vL-Sch-kernel} for Martin-Löf
randomness is a new result\textemdash although it is not quite as
useful in the Martin-Löf randomness case since no-randomness-from-nothing
(the second conclusion of Corollary~\ref{cor:rand-pres-nrfn}) holds
for Martin-Löf randomness without any computability assumptions on
the conditional probability $\mu(\cdot\mid T=\cdot)$. (For example,
Hoyrup and Rojas \cite[Prop.~5]{Hoyrup.Rojas:2009b} and Bienvenu
and Porter \cite[Thm.~3.5]{Bienvenu.Porter:2012}. Also see the survey
by Bienvenu, Hoyrup, and Shen \cite[Thm.~5]{Bienvenu.Hoyrup.Shen:2016}.)
\end{rem}

\section{\label{sec:useful-characterization}A useful characterization of
Schnorr randomness for noncomputable measures}

So far we have shown our definition of Schnorr randomness for noncomputable
measures enjoys the properties one would want in such a randomness
notion. However, the fact remains that our definition is difficult
to work with for two reasons:
\begin{enumerate}
\item Many randomness results for arbitrary measures and arbitrary spaces
require reasoning which is not uniform in the measure\textemdash but
is uniform in the Cauchy names for the measure.
\item Many results for Schnorr randomness in the literature use sequential
tests or martingale tests, while our definition relies on integral
tests.
\end{enumerate}
We will address both of these problems in this section. First we give
a characterization of Schnorr randomness for noncomputable measures
where we only require uniformity in the Cauchy names for the measure
and the oracle. Then we apply this to give sequential test and martingale
test characterizations of Schnorr randomness. We also apply this to
show that, for noncomputable measures, Schnorr randomness is still
stronger than Kurtz randomness.

These results (and their proofs) suggest that most results for Schnorr
randomness can be relativized.

\subsection{\label{subsec:SR-rel-to-name}Randomness relative to the name for
the measure}

The key to overcoming our two major difficulties is to consider randomness,
not relative to the measure, but relative to a \emph{name} for the
measure. This has already been done for Martin-Löf randomness. This
approach originated in work of Reimann \cite[\S 2.6]{Reimann:2008vn}
and Reimann and Slaman \cite[\S 3.1]{Reimann.Slaman:2015}. They defined
$x\in\{0,1\}^{\mathbb{N}}$ to be Martin-Löf $\mu$-random if there
is some Cauchy name $h$ for $\mu$ such that $x\notin\bigcap_{n}U_{n}$
for all sequential $\mu$-tests $(U_{n})_{n\in\mathbb{N}}$ computable
from $h$. Day and Miller \cite[Thm.~1.6]{Day.Miller:2013} showed
that this definition is equivalent to the Levin-Gács definition using
uniform tests (as in Section~\ref{sec:SR-noncomp}).\footnote{There are many differences between the Reimann-Slaman and Levin-Gács
definitions. Reimann and Slaman use sequential tests while Levin and
Gács use integral tests. It is well-known how to handle this. The
Reimann-Slaman definition is non-uniform while Levin's definition
is uniform. This is also easily handled (in contrast to Schnorr randomness
where the uniform and non-uniform versions are different). The significant
difference between the two definitions is that the Reimann-Slaman
definition uses Cauchy names while the Levin-Gács definition directly
uses the measure $\mu$. Bridging this difference is the major contribution
of the Day-Miller result which we focus on in this section.}  Day and Miller's proof naturally extends to effectively compact
metric spaces as shown in Bienvenu, Gács, Hoyrup, Rojas, and Shen
\cite[Thm.~5.36, Lem.~7.21]{Bienvenu.Gacs.Hoyrup.ea:2011}. (Moreover,
by adapting our proof of Theorem~\ref{thm:SR-name} below, one can
see that Day and Miller's result holds for all computable metric spaces.)
\begin{thm}[Day and Miller]
\label{thm:MLR-name}The following are equivalent:

\begin{enumerate}
\item $x\in\MLR{\mu}{}$.
\item $x\in\MLR{\mu}h$ for some Cauchy name $h\in\mathbb{N}^{\mathbb{N}}$
for $\mu$.
\end{enumerate}
\end{thm}
The reason this result is not trivial is that there are measures $\mu$
for which $\mu$ cannot compute any of its Cauchy names. In other
words, $\mu$ has no Cauchy name of least Turing degree. (See Day
and Miller \cite{Day.Miller:2013} for discussion.)

Here we give a similar result for Schnorr randomness, but using a
very different proof from that of Day and Miller. In particular, our
proof also applies to noncompact spaces $\mathbb{X}$.
\begin{thm}
\label{thm:SR-name}Let $a\in\mathbb{A}$, $\mu\in\finitemeas(\mathbb{X})$,
and $x\in\mathbb{X}$. The following are equivalent.

\begin{enumerate}
\item $x\in\SR{\mu}h$ for some Cauchy name $h\in\mathbb{N}^{\mathbb{N}}$
for the pair $(a,\mu)$.
\item $x\in\SR{\mu}h$ for some $h\in\mathbb{N}^{\mathbb{N}}$ which uniformly
computes $(a,\mu)$ as in Definition~\ref{def:uniform-comp}.
\item $x\in\SR{\mu}a$.
\end{enumerate}
\end{thm}
\begin{proof}
(1) implies (2) since each Cauchy name $h$ for $(a,\mu)$ uniformly
computes $(a,\mu)$. (2) implies (3) by Proposition~\ref{prop:reducible-oracles}.

Now we show the difficult direction, (3) implies (1). Informally,
the main idea is that $x\in\SR{\mu}a$ iff $x\in\SR{\mu}h$ for some
``random'' name $h$ for $(a,\mu)$. This ``randomness'' is achieved
by finding a probability measure $\xi^{a,\mu}$ on $\mathbb{N}^{\mathbb{N}}$,
computable uniformly in $(a,\mu)$, which is supported on the Cauchy
names for $(a,\mu)$.

Formally, we construct $\xi^{a,\mu}$ as follows. Let $\{b_{n}\}_{n\in\mathbb{N}}$
denote the basic points of the computable metric space $\mathbb{A}\times\finitemeas(\mathbb{X})$
with metric $d$. For each finite sequence $\sigma\in\mathbb{N}^{*}$
we construct the measure $\xi^{a,\mu}[\sigma]$ of the cylinder set
$[\sigma]=\{h\in\mathbb{N}^{\mathbb{N}}:h\upharpoonright|\sigma|=\sigma\}$
to be as follows 
\[
\xi^{a,\mu}[\sigma]=\prod_{i=0}^{|\sigma|}\frac{2^{-\sigma(i)}\left(2^{-(i+1)}\dotminus d(b_{\sigma(i)},(a,\mu))\right)}{\sum_{n\in\mathbb{N}}2^{-n}\left(2^{-(i+1)}\dotminus d(b_{n},(a,\mu))\right)}.
\]
(Here $\dotminus$ is truncated subtraction, $x\dotminus y:=\max\{x-y,0\}$.)
The term inside the product represents the probability of choosing
$h(i)$, the $i$th value in the Cauchy name $h$, independently of
all the other values. This $\xi^{a,\mu}$ is a probability measure
uniformly computable from $a,\mu$. Also, $\xi^{a,\mu}$ is supported
on the set of $h\in\mathbb{N}^{\mathbb{N}}$ such that $d(b_{h(i)},(a,\mu))\leq2^{-(i+1)}$
for all $i$, so by Proposition~\ref{prop:support}, any $h\in\SR{\xi^{a,\mu}}{}$
is a Cauchy name for $(a,\mu)$.

Fix $a_{0},\mu_{0}$ and choose $x_{0}\in\SR{\mu_{0}}{a_{0}}$. Since
$\xi^{a_{0},\mu_{0}}$ is uniformly computable in $a_{0},\mu_{0}$
we have that $x_{0}\in\SR{\mu_{0}}{a_{0},\xi^{a_{0},\mu_{0}}}$ by
Proposition~\ref{prop:reducible-oracles}. Choose $h_{0}\in\SR{\xi^{a_{0},\mu_{0}}}{a_{0},\mu_{0},x_{0}}$.
The corollary to Van Lambalgen's Theorem (Corollary~\ref{cor:vL_SR-symmetric})
states that
\[
\left(x_{0}\in\SR{\mu_{0}}{a_{0},\xi^{a_{0},\mu_{0}}}\ \text{and}\ h_{0}\in\SR{\xi^{a_{0},\mu_{0}}}{a_{0},\mu_{0},x_{0}}\right)\quad\text{iff}\quad\left(h_{0}\in\SR{\xi^{a_{0},\mu_{0}}}{a_{0},\mu_{0}}\ \text{and}\ x_{0}\in\SR{\mu_{0}}{a_{0},\xi^{a_{0},\mu_{0}},h_{0}}\right).
\]
In particular, this gives us $x_{0}\in\SR{\mu_{0}}{a_{0},h_{0},\xi^{a_{0},\mu_{0}}}$.
Last, the pair $(a_{0},\xi^{a_{0},\mu_{0}})$ is uniformly computable
in $h_{0}$, so $x_{0}\in\SR{\mu_{0}}{h_{0}}$ by Proposition~\ref{prop:reducible-oracles}.
\end{proof}

\subsection{\label{subsec:unif-seq-tests-and-SLLN}Uniform Schnorr sequential
tests with applications}

Now let us consider the uniform version of Schnorr sequential tests.
\begin{defn}
\label{def:unif-seq-test}A \emph{uniform Schnorr sequential test}
restricted to an effectively closed set $K\subseteq\mathbb{A}\times\finitemeas(\mathbb{X})$
is a family $\{U_{\mu}^{n;a}\}_{n\in\mathbb{N},(a,\mu)\in K}$ of
open subsets of $\mathbb{X}$ such that $U_{\mu}^{n;a}$ is $\Sigma_{1}^{0}[n,a,\mu]$,
$\mu(U_{\mu}^{n,a})\leq2^{-n}$, and $a,n,\mu\mapsto\mu(U_{\mu}^{n;a})$
is computable. 
\end{defn}
\begin{prop}
\label{prop:seq-test-easy}Let $\{U_{\mu}^{n;a}\}_{n\in\mathbb{N},(a,\mu)\in K}$
be a uniform sequential test. For all $(a,\mu)\in K$ and $x\in\mathbb{X}$,
if $x\in\SR{\mu}a$, then $x\notin\bigcap_{n\in\mathbb{N}}U_{\mu}^{n;a}$
.
\end{prop}
\begin{proof}
Fix $a_{0},\mu_{0},x_{0}$ such that $x_{0}\in\bigcap_{n\in\mathbb{N}}U_{\mu_{0}}^{n;a_{0}}$.
The function $t_{\mu}^{a}=\sum_{n\in\mathbb{N}}\mathbf{1}_{U_{\mu}^{n;a}}$
is a uniform Schnorr integral test with $t(x_{0})=\sum_{n\in\mathbb{N}}\mathbf{1}_{U_{\mu_{0}}^{n;a_{0}}}=\sum_{n\in\mathbb{N}}1=\infty$.
\end{proof}
However, the converse to Proposition~\ref{prop:seq-test-easy} generally
fails. (See Subsection~\ref{subsec:uniform-sequential-tests}.) Nonetheless,
when combined with Theorem~\ref{thm:SR-name}, uniform Schnorr sequential
tests are a useful tool. They are especially useful when combined
with the following two lemmas, which are some of the most important
lemmas in computable measure theory. Such lemmas allow one to apply
Cantor-space-like reasoning to arbitrary computable metric spaces.
(See, for example, Hoyrup and Rojas \cite{Hoyrup.Rojas:2009}.)
\begin{lem}
\label{lem:a.e.-basis}Given a measure $\mu\in\finitemeas(\mathbb{X})$
with Cauchy name $h\in\mathbb{N}$, a basic point $x_{i}\in\mathbb{X}$,
and two positive rationals $q_{1}<q_{2}$, one can effectively (uniformly
in $h,i,q_{1},q_{2}$) find a radius $r\in[q_{1},q_{2}]$ such that
$\mu\{x\in\mathbb{X}:d_{\mathbb{X}}(x,x_{i})=r\}=0$.
\end{lem}
\begin{proof}
Follow a basic diagonalization argument. Start by searching for rationals
$a$ and $b$ such that $q_{1}\leq a<b\leq q_{2}$ and both $|b-a|$
and $\mu\{x\in\mathbb{X}:a\leq d_{\mathbb{X}}(x,x_{i})\leq b\}$ are
small. Then repeat, replacing $q_{1}$ and $q_{2}$ with $a$ and
$b$. Let $r$ be the limit. Such a search can be done effectively
in the Cauchy name $h$ of $\mu$.
\end{proof}
If $r$ is as in the previous lemma, call $B(x_{i},r)$ an \emph{$(h,\mu)$-basic open ball}
of $\mathbb{X}$ and call $\overline{B}(x_{i},r)$ an \emph{$(h,\mu)$-basic closed ball}.
By enumerating the triples $(i,q_{1},q_{2})$ we have an enumeration
of $(h,\mu)$-basic open balls $\{B_{\mu}^{n;h}\}$ and $(h,\mu)$-basic
closed balls $\{\overline{B}_{\mu}^{n;h}\}_{n\in\mathbb{N}}$. (Note
that the radius of the $n$th ball depends uniformly on $h,\mu$.)
We have the following convenient facts about $(h,\mu)$-basic balls.
\begin{lem}
\label{lem:a.e.-basis-facts}Consider the effectively closed set 
\[
K=\{(h,\mu)\in\mathbb{N}^{\mathbb{N}}\times\finitemeas(\mathbb{X}):h\text{ is a Cauchy name for }\mbox{\ensuremath{\mu}}\}.
\]

\begin{enumerate}
\item If $U\subseteq K\times\mathbb{X}$ is $\Sigma_{1}^{0}$, then there
is a computable map $f\colon K\to\{0,1\}^{\mathbb{N}}$ such that
$U_{\mu}^{h}=\bigcup_{n\in f(h,\mu)}B_{\mu}^{n;h}=\bigcup_{n\in f(h,\mu)}\overline{B}_{\mu}^{n;h}$.
\item If $A_{\mu}^{h}$ is a finite Boolean combination of $(h,\mu)$-basic
open balls $B_{\mu}^{0;h},\ldots,B_{\mu}^{n-1;h}$, then $h,\mu\mapsto\mu(A_{\mu}^{h})$
is computable (uniformly in the code for the Boolean combination).
\end{enumerate}
\end{lem}
\begin{proof}
(1) Given $U_{\mu}^{h}$, effectively find a computable sequence of
basic open balls $B(x_{i},r)$ such that $U_{\mu}^{h}=\bigcup_{i}B(x_{i},r)$.
(This is computable since we are working with the Cauchy name $h\in\mathbb{N}^{\mathbb{N}}$.)
Then we can easily replace each basic open ball $B(x_{i},r)$ with
a sequence of $(h,\mu)$-basic open or closed balls with the same
center $x_{i}$ but with smaller radii.

(2) Consider just the balls $B_{\mu}^{n;h}$ and $\overline{B}_{\mu}^{n;h}$
which are respectively $\Sigma_{1}^{0}[n,h,\mu]$ and $\Pi_{1}^{0}[n,h,\mu]$.
The map $n,h,\mu\mapsto\mu(B_{\mu}^{n;h})$ is lower semicomputable
and the map $n,h,\mu\mapsto\mu(\overline{B}_{\mu}^{n;h})$ is upper
semicomputable (Lemma~\ref{lem:Portmanteau}). Since $\mu(B_{\mu}^{n;h})=\mu(\overline{B}_{\mu}^{n;h})$,
the maps are computable. The same idea holds for finite Boolean combinations.
\end{proof}
Now we apply the above to show that for noncomputable measures, Schnorr
randomness is still stronger than Kurtz randomness.
\begin{prop}
\label{prop:SR-implies-KR}If $x\in\SR{\mu}a$ then $x\in\KR{\mu}a$.
\end{prop}
\begin{proof}
Fix $a_{0},\mu_{0},x_{0}$ and assume $x_{0}\in\SR{\mu_{0}}{a_{0}}$.
By Theorem~\ref{thm:SR-name}, there is some Cauchy name $h_{0}$
for $(a_{0},\mu_{0})$ such that $x_{0}\in\SR{\mu_{0}}{h_{0},a_{0}}$.
It is enough to show that $x_{0}\in\KR{\mu_{0}}{h_{0},a_{0}}$ since
clearly $\KR{\mu_{0}}{h_{0},a_{0}}\subseteq\KR{\mu_{0}}{a_{0}}$.
Consider an effectively closed set $P\subseteq\mathbb{N}^{\mathbb{N}}\times\mathbb{A}\times\finitemeas(\mathbb{X})\times\mathbb{X}$
such that $\mu(P_{\mu}^{h,a})=0$ for all $h,a,\mu$. We must show
$x_{0}\notin P_{\mu_{0}}^{a_{0},h_{0}}$.

Let $U_{\mu}^{h,a}$ be the complement of $P_{\mu}^{h,a}$. By Lemma~\ref{lem:a.e.-basis-facts}(1)
we can enumerate a sequence of $(h,\mu)$-basic closed balls $\overline{B}_{\mu}^{i;h}$
such that $U_{\mu}^{h,a}=\bigcup_{i}\overline{B}_{\mu}^{i;h}$. Find
a finite subsequence of these balls whose union has measure $>\|\mu\|-2^{-n}$.
Let $V_{\mu}^{n;h,a}$ be the complement of this finite union. Then
$V_{\mu}^{n;h,a}$ is $\Sigma_{1}^{0}[n,h,a,\mu]$, $P_{\mu}^{h,a}\subseteq V_{\mu}^{n;h,a}$,
and $\mu(V_{\mu}^{n;h,a})\leq2^{-n}$. By Lemma~\ref{lem:a.e.-basis-facts}(2),
$n,h,a,\mu\mapsto\mu(V_{\mu}^{n;h,a})$ is computable. Hence $\{V_{\mu}^{n;h,a}\}_{n\in\mathbb{N},(h,a,\mu)\in K}$
(where is $K$ is the set of all $(h,a,\mu)$ where $h$ is a Cauchy
name for $(a,\mu)$) is a uniform Schnorr sequential test. By Proposition~\ref{prop:seq-test-easy},
$x_{0}\notin\bigcap_{n}V_{\mu_{0}}^{n;h_{0},a_{0}}$. Therefore, $x_{0}\notin P_{\mu_{0}}^{a_{0},h_{0}}$
as desired.
\end{proof}
\begin{rem}
Our proof of Proposition~\ref{prop:SR-implies-KR} follows the usual
sequential test proof for computable measures (folklore, compare with
Hoyrup and Rojas \cite[Lem.~6.2.1]{Hoyrup.Rojas:2009}). We only needed
to check that the steps are uniform in the Cauchy name of the measure.
Now, with some ingenuity one could alternately find a direct proof
which uses integral tests and avoids Cauchy names. (We leave this
as an exercise for the reader. Hint:\ Use bump functions.) Nonetheless,
our point is that one does not need to be ingenious. In general, most
proofs concerning Schnorr randomness can be naturally relativized
as in the above example.
\end{rem}
When we say that Schnorr randomness is ``stronger'' than Kurtz randomness,
we mean that 
\[
\{(a,\mu,x)\mid x\in\SR{\mu}a\}\subsetneqq\{(a,\mu,x)\mid x\in\KR{\mu}a\}
\]
Inclusion follows from Proposition\ \ref{prop:SR-implies-KR}. Nonequality
is already known for the computable case. For example, it is well-known
that for the fair-coin measure every weak 1-generic is Kurtz random
but not Schnorr random \cite[\S 8.11.2]{Downey.Hirschfeldt:2010}.
However, for certain special measures $\mu$ it is the case that $\KR{\mu}{}=\SR{\mu}{}$.
By Proposition~\ref{prop:support}, every measure $\mu$ with finite
support\textemdash for example a Dirac measure $\delta_{x}$\textemdash satisfies
$\KR{\mu}{}=\SR{\mu}{}=\MLR{\mu}{}$. Moreover, in Section~\ref{sec:conclusion}
we discuss noncomputable ``neutral measures'' $\mu$ such that \emph{every point}
is Martin-Löf $\mu$-random. These measures also satisfy $\KR{\mu}{}=\SR{\mu}{}=\MLR{\mu}{}$.

\subsection{\label{subsec:seq-test-rel-to-name}Characterizing Schnorr randomness
for noncomputable measures via sequential tests}

The previous section points to the following characterization of Schnorr
randomness for noncomputable measures via sequential tests.
\begin{thm}
\label{thm:SR-seq-test}Consider the effectively closed set 
\[
K=\{(h,\mu)\in\mathbb{N}^{\mathbb{N}}\times\finitemeas(\mathbb{X}):h\text{ is a Cauchy name for }\mbox{\ensuremath{\mu}}\}.
\]
For all $(h_{0},\mu_{0})\in K$ and $x_{0}\in\mathbb{X}$, the following
are equivalent:

\begin{enumerate}
\item $x_{0}\in\SR{\mu_{0}}{h_{0}}$.
\item $x_{0}\notin\bigcap_{n\in\mathbb{N}}U_{\mu_{0}}^{n;h_{0}}$ for all
uniform Schnorr sequential tests $\{U_{\mu}^{n;h}\}_{n\in\mathbb{N},(h,\mu)\in K}$.
\end{enumerate}
\end{thm}
\begin{proof}
The direction (1) implies (2) follows from Proposition~\ref{prop:seq-test-easy}.

For (2) implies (1), assume $x_{0}\notin\SR{\mu_{0}}{h_{0}}$. Then
there is a uniform Schnorr integral test $\{t_{\mu}^{h}\}_{(h,\mu)\in K}$
such that $t_{\mu_{0}}^{h_{0}}(x_{0})=\infty$ and $\int_{{}}t_{\mu}^{h}\,d\mu\leq1$.
Our goal is to find a uniform Schnorr sequential test $\{U_{\mu}^{n;h}\}_{n\in\mathbb{N},(h,\mu)\in K}$
such that $x_{0}\in\bigcap_{n}U_{\mu_{0}}^{n;h_{0}}.$

First, let us just consider the sets 
\[
U_{\mu}^{n;h}=\{x:t_{\mu}^{h}(x)>2^{n}\}.
\]
The set $U_{\mu}^{n;h}$ is $\Sigma_{1}^{0}[n,h,\mu]$ and by Markov's
inequality,
\[
\mu(U_{\mu}^{n;h})\leq2^{-n}\int_{{}}t_{\mu}^{h}\,d\mu\leq2^{-n}.
\]
Unfortunately, $\mu(U_{\mu}^{n;h})$ may not be computable. This is
fixed by the following claim.

\begin{claim*}
From $h$ and $\mu$ (where $h$ is a Cauchy name for $\mu$) we can
uniformly compute some value $c(n,h,\mu)\approx2^{n+1}$ such that
$n,h,\mu\mapsto\mu\{x:t_{\mu}^{h}(x)>c(n,h,\mu)\}$ is computable
and $\mu\{x:t_{\mu}^{h}(x)>c(n,h,\mu)\}\leq2^{-n}$ for all $n,h,\mu$.
\end{claim*}
\begin{proof}[Proof of claim]
(For more details, see Miyabe and Rute \cite{Miyabe.Rute:2013},
specifically the claim in the proof of the Key Lemma.) If we can compute
a value $c(n,h,\mu)$ such that $\mu\{x:t_{\mu}^{h}(x)=c(n,h,\mu)\}=0$,
then $h,a,\mu\mapsto\mu\{x:t_{\mu}^{h}(x)>c(n,h,\mu)\}$ is computable.
To see this, first note that $\mu\{x:t_{\mu}^{h}(x)>c(n,h,\mu)\}$
is lower semicomputable in the parameters (Lemma~\ref{lem:Portmanteau}).
We can also show that $\mu\{x:t_{\mu}^{h}(x)\geq c(n,h,\mu)\}$ is
upper semicomputable as follows. Pick $\varepsilon>0$ and approximate
$t_{\mu}^{h}$ from below with a continuous function $f$ which is
close enough to $t_{\mu}^{h}$ in the $L^{1}$-norm such that 
\[
\mu\{x:t_{\mu}^{h}(x)\geq c(n,h,\mu)\}\approx\mu\{x:f(x)\geq c(n,h,\mu)-\varepsilon\},
\]
The right-hand-side is upper semicomputable in the parameters (Lemma~\ref{lem:Portmanteau}).
Therefore, $\mu\{x:t_{\mu}^{h}(x)>c(n,h,\mu)\}=\mu\{x:t_{\mu}^{h}(x)\geq c(n,h,\mu)\}$
is computable.

Now, we need to compute $c(n,h,\mu)$. For a sufficiently small $\varepsilon$,
break up $[2^{n+1}-\varepsilon,2^{n+1}+\varepsilon]$ into small rational
intervals $[a,b]$. Then use $h$ to approximate $t_{\mu}^{h}$ with
some continuous function $f^{h}$ close enough to $t_{\mu}^{h}$ so
that 
\[
\mu(\{x\in\mathbb{X}:t_{\mu}^{h}(x)\in[a,b]\})\approx\mu(\{x\in\mathbb{X}:f^{h}\in[a,b]\}).
\]
The right-hand-side is upper semicomputable in $h$. Putting this
all together, one may use $h$ to search for some small interval $[a,b]$
such that $\mu(\{x\in\mathbb{X}:t_{\mu}^{h}(x)\in[a,b]\})$ is sufficiently
small. Continuing like this, one can use $h$ to uniformly compute
$c(n,h,\mu)$.

(Note: We need the Cauchy name $h$ in order to perform this search
for small intervals \emph{uniformly}. Different estimates of the
value of $\mu(\{x\in\mathbb{X}:t_{\mu}^{h}(x)\in[a,b]\})$ will lead
us to different limits $c(n,h,\mu)$. To make this computation uniform,
we need a Cauchy name $h$ to uniformly compute these estimates.)

This completes the proof of the claim.
\end{proof}
Finally, let $V_{\mu}^{n;h}=\{x:t_{\mu}^{h}(x)>c(n,h,\mu)\}.$ This
is the desired uniform Schnorr sequential test.
\end{proof}
There is nothing special here about sequential tests. One may apply
these principles to other characterizations of Schnorr randomness,
for example Schnorr's characterization using martingale tests.
\begin{defn}
(On $\{0,1\}^{\mathbb{N}}$.) A \emph{uniform Schnorr martingale test}
is a pair of family of pairs $\{\nu_{\mu}^{a},f_{\mu}^{a}\}_{(a,\mu)\in K}$
such that $K\subseteq\mathbb{A}\times\finitemeas(\{0,1\}^{\mathbb{N}})$,
$\nu_{\mu}^{a}$ is a measure on $\{0,1\}^{\mathbb{N}}$ uniformly
computable in $a$ and $\mu$, and $f_{\mu}^{a}\colon\mathbb{N}\rightarrow\mathbb{R}$
is an unbounded nondecreasing function uniformly computable in $a$
and $\mu$.
\end{defn}
We call this a martingale test because $\nu_{\mu}^{a}[\sigma]/\mu[\sigma]$
is a martingale. That is 
\[
\left(\frac{\nu_{\mu}^{a}[\sigma0]}{\mu[\sigma0]}\right)\mu[\sigma0]+\left(\frac{\nu_{\mu}^{a}[\sigma1]}{\mu[\sigma1]}\right)\mu[\sigma1]=\left(\frac{\nu_{\mu}^{a}[\sigma]}{\mu[\sigma]}\right)\mu[\sigma].
\]

\begin{prop}
\label{prop:SR-mart} Consider the space $\{0,1\}^{\mathbb{N}}$ and
the effectively closed set 
\[
K=\{(h,\mu)\in\mathbb{N}^{\mathbb{N}}\times\finitemeas(\{0,1\}^{\mathbb{N}}):h\text{ is a Cauchy name for }\mbox{\ensuremath{\mu}}\}.
\]
For all $(h_{0},\mu_{0})\in K$ and $x_{0}\in\{0,1\}^{\mathbb{N}}$
, the following are equivalent:

\begin{enumerate}
\item $x_{0}\in\SR{\mu_{0}}{h_{0}}$.
\item For every uniform Schnorr martingale test $\{\nu_{\mu}^{h},f_{\mu}^{h}\}_{(h,\mu)\in K}$
we have
\[
\forall n\ \mu_{0}[x_{0}\upharpoonright n]>0\quad\text{and}\quad\forall^{\infty}n\ \frac{\nu^{h_{0}}[x_{0}\upharpoonright n]}{\mu_{0}[x_{0}\upharpoonright n]}\leq f_{\mu_{0}}^{h_{0}}(n).
\]
\end{enumerate}
\end{prop}
\begin{proof}
Just follow the usual proof, as say found in Downey and Hirschfeldt~\cite[Thm.~7.1.7]{Downey.Hirschfeldt:2010},
using the tricks from the proof of Theorem~\ref{thm:SR-seq-test}
as needed. (Also, see Miyabe and Rute \cite[\S 6.3]{Miyabe.Rute:2013}
for some discussion and tricks related to uniform martingale tests.)
\end{proof}

\section{\label{sec:Other-bad-defs}Alternative definitions of Schnorr randomness
for noncomputable measures}

In this section, we explore alternate definitions of Schnorr randomness
for noncomputable measures which have appeared in the literature or
the folklore. All of these definitions are distinct from our definition.
We argue that our definition has more desirable properties. 

(The one definition we do not address is Schnorr's definition of Schnorr
randomness for a noncomputable Bernoulli measures found in Sections
24 and 25 of his book \cite{Schnorr:1971rw}. While it seems that
his definition is equivalent to ours\textemdash for Bernoulli measures\textemdash{}
the details are long and beyond the scope of this paper.)

\subsection{\label{subsec:naive-def}Non-uniform definition}

We must address what, for most computability theorists, would be the
natural extension of Schnorr randomness to both noncomputable oracles
and noncomputable measures.
\begin{defn}
Fix $a\in\mathbb{A}$ and $\mu\in\probmeas(\mathbb{X})$. A \emph{non-uniform Schnorr sequential $\mu$-test relative to $a$}
is a sequence $(U_{(\mu)}^{n;(a)})_{n\in\mathbb{N}}$ of open sets
such that $U_{(\mu)}^{n;(a)}$ is $\Sigma_{1}^{0}[n,a,\mu]$, $\mu(U_{(\mu)}^{n;(a)})\leq2^{-n}$,
and $\mu(U_{(\mu)}^{n;(a)})$ is computable from $a,\mu$.\footnote{We use the notation $(a)$ and $(\mu)$ here to remind the reader
that the test $(U_{(\mu)}^{n;(a)})_{n\in\mathbb{N}}$ depends non-uniformly
on $a$ and $\mu$. In particular, when we say that $\mu(U_{(\mu)}^{n;(a)})$
is computable from $a,\mu$, we mean that there is a \emph{partial}
computable map $f\colon{\subseteq{}}\mathbb{A}\times\probmeas(\mathbb{X})\to\mathbb{R}$
such that $\mu(U_{(\mu)}^{n;(a)})=f(a,\mu)$. It may not be possible
to extend $(U_{(\mu)}^{n;(a)})_{n\in\mathbb{N}}$ to a uniform sequential
test $(U_{\mu}^{n;a})_{n\in\mathbb{N},a\in\mathbb{A},\mu\in\mathcal{M}_{1}(\mathbb{X})}$.} A point $x\in\mathbb{X}$ is \emph{non-uniformly Schnorr $\mu$-random relative to the oracle $a$}
if $x\notin\bigcap_{n\in\mathbb{N}}U_{n}$ for all non-uniform Schnorr
sequential $\mu$-tests relative to $a$.
\end{defn}
(Here we used the sequential test characterization of Schnorr randomness,
but any of the standard characterizations, including integral tests,
would give the same result.) It is easy to see that this non-uniform
definition is at least as strong as ours. Therefore, for every $x$
which is uniformly Schnorr $\mu$-random relative to $a$ we have
that $x\in\SR{\mu}a$. While to our knowledge, there has been no investigations
using this definition of Schnorr randomness for noncomputable measures,
there are a number of books and papers which use this non-uniform
definition of Schnorr randomness relative to a noncomputable oracle
(for example the books \cite{Nies:2009,Downey.Hirschfeldt:2010}).\footnote{These books and papers do not call this definition ``non-uniform.''
They usually just say ``$x$ is $a$-Schnorr random'' or ``$x\in\SR{}a$.''}  This is especially true of papers concerning lowness for randomness
\cite{Kjos-Hanssen.Nies.Stephan:2005,Franklin:2010b,Bienvenu.Miller:2012}.

Yu \cite{Yu:2007} showed that when $(\mathbb{X},\mu)$ is the fair
coin measure $(\{0,1\}^{\mathbb{N}},\lambda)$, this non-uniform definition
does not satisfy Van Lambalgen's Theorem. (The result is also implicit
in the proof of Theorem 5 in Merkle, Miller, Nies, Reimann, and Stephan
\cite{Merkle.Miller.Nies.ea:2006}.)
\begin{prop}[Merkle et. al \cite{Merkle.Miller.Nies.ea:2006}, Yu \cite{Yu:2007}]
\label{prop:vL-failure}There is a pair of sequences $x,y\in\{0,1\}^{\mathbb{N}}$
such that $(x,y)$ is Schnorr random on $\lambda\otimes\lambda$ but
$x$ is not non-uniformly Schnorr $\lambda$-random relative to $y$.
\end{prop}
This shows that the non-uniform definition is different from ours
(and has less desirable properties). Moreover, this difference extends
to randomness for noncomputable measures as follows.
\begin{prop}
\label{prop:naive-different}There is a noncomputable probability
measure $\mu$ on $\{0,1\}^{\mathbb{N}}$ and a point $z\in\{0,1\}^{\mathbb{N}}$
such that $z\in\SR{\mu}{}$, but $z$ is not non-uniformly Schnorr
$\mu$-random.
\end{prop}
\begin{proof}
Let $x,y\in\{0,1\}^{\mathbb{N}}$ be from Proposition~\ref{prop:vL-failure}.
By our generalization of Van Lambalgen's Theorem (Theorem~\ref{thm:vL-sch}),
we have $x\in\SR{\lambda}y$. However, by Proposition~\ref{prop:vL-failure},
$x$ is not non-uniformly Schnorr $\lambda$-random relative to $y$. 

Now, let $z=(x,y)$ and $\mu=\lambda\otimes\delta_{y}$. Since $\{0,1\}^{\mathbb{N}}$
is computably homeomorphic to $\{0,1\}^{\mathbb{N}}\times\{0,1\}^{\mathbb{N}}$,
we can view $z$ as a point in $\{0,1\}^{\mathbb{N}}$ and $\mu$
as a measure on $\{0,1\}^{\mathbb{N}}$. By Proposition~\ref{prop:Dirac},
$z\in\SR{\mu}{}$. Similarly, it is easy to see that $z$ is not non-uniformly
Schnorr $\mu$-random since one can use $\mu$ to compute $y$, and
$y$ to compute a non-uniform Schnorr sequential $\lambda$-test $U_{\lambda}^{n;(y)}$
covering $x$. Then $V_{(\mu)}^{n}=U_{\lambda}^{n;(y)}\times\{0,1\}^{\mathbb{N}}$
is a non-uniform Schnorr sequential $\mu$-test covering $z$.
\end{proof}
(With some more work, we could have even constructed a Bernoulli measure
$B_{p}$ on $\{0,1\}^{\mathbb{N}}$ satisfying Proposition~\ref{prop:naive-different},
however the proof is beyond the scope of this paper.)

\subsection{\label{subsec:Schnorr-Fuchs-def}Schnorr-Fuchs definition}

Schnorr and Fuchs \cite[Def~3.1]{Schnorr.Fuchs:1977} (also Schnorr
\cite[Def.~3.4]{Schnorr:1977}), introduced a notion of Schnorr randomness
for noncomputable measures on $\{0,1\}^{\mathbb{N}}$. Their definition
is based on Schnorr's martingale characterization of Schnorr randomness.
\begin{defn}[Schnorr and Fuchs]
\label{def:Sch-mart-def}Given a measure $\mu\in\probmeas(\{0,1\}^{\mathbb{N}})$,
say that $x\in\{0,1\}^{\mathbb{N}}$ is \emph{Schnorr-Fuchs $\mu$-random}
if there are no computable measures $\nu$ and no unbounded nondecreasing
computable functions $f\colon\mathbb{N}\rightarrow\mathbb{R}$ such
that
\[
\exists^{\infty}n\ \frac{\nu[x\upharpoonright n]}{\mu[x\upharpoonright n]}\geq f(n).
\]
\end{defn}
Compare this definition to our martingale characterization of $\SR{\mu}{}$
in Proposition~\ref{prop:SR-mart}. The definitions are similar except
in Proposition~\ref{prop:SR-mart}, $\nu$ and $f$ may depend on
$\mu$. It follows that every $x\in\SR{\mu}{}$ is Schnorr-Fuchs $\mu$-random.

However the converse fails. The Schnorr-Fuchs definition is ``blind,''
in that the test does not use the measure as an oracle. (See Section~\ref{sec:conclusion}.)
\begin{prop}
There is a measure $\mu$ and some $x\in\{0,1\}^{\mathbb{N}}$ such
that $x$ is Schnorr-Fuchs $\mu$-random but $x\notin\SR{\mu}{}$.
\end{prop}
\begin{proof}
Let $x,y\in\{0,1\}^{\mathbb{N}}$ be Schnorr $\lambda$-random sequences
which are equal except for the first bit: $x$ starts with $0$, and
$y$ starts with $1$. Let $\mu$ be the measure where $\mu[0]=\mu[1]=\frac{1}{2}$
and the mass on $[0]$ is distributed uniformly\textemdash as in the
fair coin measure\textemdash and the mass on $[1]$ is entirely concentrated
on $y$\textemdash so $\mu(\{y\})=\frac{1}{2}$.

Firstly, for all computable $\nu$ and $f$ we have 
\[
\forall^{\infty}n\ \frac{\nu[x\upharpoonright n]}{\mu[x\upharpoonright n]}=\frac{\nu[x\upharpoonright n]}{\lambda[x\upharpoonright n]}\leq f(n)
\]
since $\mu[x\upharpoonright n]=\lambda[x\upharpoonright n]$ for $n\geq1$
and since $x$ is Schnorr $\lambda$-random. Therefore $x$ is Schnorr-Fuchs
$\mu$-random.

Moreover, $\mu$ uniformly computes $y$ as in Definition~\ref{def:uniform-comp}.
(The atom $y$ can be computed by looking for all finite strings $\sigma\in\{0,1\}^{*}$
such that $\mu[1\sigma]\geq\frac{3}{8}$\textemdash restricting the
computation to the effectively closed set $K$ of measures with such
a property.) Therefore, $\mu$ also uniformly computes $x$ and it
is easy to see that $x\notin\SR{\mu}{}$ (since the measure is uniform
around $x$).
\end{proof}
\begin{rem}
Unlike the non-uniform definition of Schnorr randomness, we conjecture
that the Schnorr-Fuchs definition agrees with our definition on a
large class of measures, e.g.~Bernoulli measures. See Section~\ref{sec:conclusion}.
\end{rem}

\subsection{\label{subsec:uniform-sequential-tests}Uniform sequential test definition}

The reader may wonder why we rely on integral tests instead of the
more common sequential tests. We defined uniform sequential tests
$\{U_{\mu}^{n}\}_{n\in\mathbb{N},\mu\in K}$ in Definition~\ref{def:unif-seq-test},
and used them in Subsections~\ref{subsec:unif-seq-tests-and-SLLN}
and \ref{subsec:seq-test-rel-to-name}. However, they are not as robust
as uniform Schnorr integral tests. Mathieu Hoyrup (unpublished, personal
communication) considered using uniform sequential tests to define
Schnorr randomness on noncomputable measures, but ran into this serious
flaw.
\begin{prop}[Hoyrup]
\label{prop:unif-seq-tests-on-reals}There are no non-trivial uniform
Schnorr sequential tests $\{U_{\mu}^{n}\}_{n\in\mathbb{N},\mu\in\probmeas(\mathbb{X})}$
on the space $\mathbb{X}=[0,1]$. Specifically, if $U_{\mu}\subseteq[0,1]$
is $\Sigma_{1}^{0}[\mu]$ such that $\mu\mapsto\mu(U_{\mu})$ is computable,
then either $U_{\mu}$ is empty for all $\mu$ or is $\mu$-measure
one for all $\mu$.
\end{prop}
\begin{proof}
Let $\{U_{\mu}\}_{\mu\in\probmeas(\mathbb{X})}$ be a family of open
sets which are $\Sigma_{1}^{0}[\mu]$. For $\mu\in\probmeas([0,1])$,
$x\in[0,1]$, and $s\in[0,1]$, consider the measure 
\[
\phi(\mu,x,s)=s\delta_{x}+(1-s)\mu.
\]

\begin{claim*}
For $s>0$, the map $\mu,x,s\mapsto\delta_{x}(U_{\phi(\mu,x,s)})\in\{0,1\}$
is continuous.
\end{claim*}
\begin{proof}[Proof of claim.]
By assumption, the map $m(\mu)=\mu(U_{\mu})$ is continuous. Since
$\phi$ is also continuous, so is the composition $m\circ\phi$, 
\[
(m\circ\phi)(\mu,x,s)=\phi(\mu,x,s)(U_{\phi(\mu,x,s)})=s\delta_{x}(U_{\phi(\mu,x,s)})+(1-s)\mu(U_{\phi(\mu,x,s)})
\]
Since both the maps $\mu,x,s\mapsto s\delta_{x}(U_{\phi(\mu,x,s)})$
and $\mu,x,s\mapsto(1-s)\mu(U_{\phi(\mu,x,s)})$ are lower semicontinuous
(Lemma~\ref{lem:Portmanteau}) and they sum to a continuous map,
they must both be continuous, proving our claim.
\end{proof}
Since we are on the connected space $[0,1]$, either this map is constant
$0$ or constant $1$. In the constant $0$ case, $x\notin U_{\phi(\mu,x,s)}$
for all $\mu,x,s$ ($s>0$). We argue that $U_{\mu}=\varnothing$
for all $\mu$. For, if $x\in U_{\mu}$, then $x\in U_{\nu}$ for
every measure $\nu$ which is sufficiently close to $\mu$, including
$\nu=\phi(\mu,x,s)$ for sufficiently small $s>0$. However, this
contradicts that $x\notin U_{\phi(\mu,x,s)}$ for $s>0$.

In the constant $1$ case, $x\in U_{\phi(\mu,x,s)}$ for all $\mu,x,s$
($s>0$). In particular, any measure $\nu$ with an atom $x$ is of
the form $\phi(\mu,x,s)$ for some $\mu$ and $s>0$. Hence, every
atom of $\nu$ is in $U_{\nu}$. Therefore, if $\nu$ is atomic (made
up entirely of atoms), then $\nu(U_{\nu})=1$. Since $\mu\mapsto\mu(U_{\mu})$
is continuous and every measure $\mu$ can be approximated by an atomic
measure, we have that $\mu(U_{\mu})=1$ for all $\mu$.
\end{proof}
The proof of Proposition~\ref{prop:unif-seq-tests-on-reals} works
for any connected space $\mathbb{X}$. For a more general space, the
same proof gives the following result, again showing there are no
nontrivial uniform Schnorr tests.
\begin{prop}[Hoyrup]
Given a computable metric space $\mathbb{X}$ and a $\Sigma_{1}^{0}[\mu]$
set $U_{\mu}\subseteq[0,1]$ such that $\mu\mapsto\mu(U_{\mu})$ is
computable, there is a clopen set $C\subseteq\mathbb{X}$ (independent
of $\mu$) such that for all $\mu$ we have $U_{\mu}\subseteq C$
and $\mu(U_{\mu})=\mu(C)$.
\end{prop}
\begin{proof}
Look at the disjoint sets: 
\begin{align*}
A_{0} & =\{(\mu,x,s):\delta_{x}(U_{\phi(\mu,x,s)})=0\}\\
A_{1} & =\{(\mu,x,s):\delta_{x}(U_{\phi(\mu,x,s)})=1\}
\end{align*}
As before, the map $\mu,x,s\mapsto\delta_{x}(U_{\phi(\mu,x,s)})\in\{0,1\}$
is continuous for $s\in(0,1]$. By the connectedness of $\probmeas(\mathbb{X})$
and $(0,1]$, if $(\mu,x,s)\in A_{i}$ for some $\mu$ and $s>0$,
then $(\mu,x,s)\in A_{i}$ for all $\mu$ and all $s>0$. In other
words, membership in $A_{i}$ depends only on $x$. Let $C$ be the
clopen set 
\[
\{x\in\mathbb{X}:\exists\mu\,\exists s>0\ \delta_{x}(U_{\phi(\mu,x,s)})=1\}.
\]
Then, just as in the previous proof, $U_{\mu}\subseteq C$ and $\mu(U_{\mu})=\mu(C)$.
\end{proof}
What if we consider uniform Schnorr sequential tests of the form $\{U_{\mu}^{n}\}_{n\in\mathbb{N},\mu\in K}$
where $K$ is an effectively closed set $K\subseteq\probmeas(\mathbb{X})$?
These next two propositions show that such an approach would not work.
Recall, a \emph{$1$-generic} $x\in\mathbb{X}$ is a point not on
the boundary of any $\Sigma_{1}^{0}$ subset of $\mathbb{X}$.
\begin{prop}
\label{prop:unif-seq-test-1-generic}Let $\{U_{\mu}^{n}\}_{n\in\mathbb{N},\mu\in K}$
be a uniform Schnorr sequential test on the space $\mathbb{X}=[0,1]$,
and let $\mu_{0}\in K$ be a 1-generic point in $\probmeas([0,1])$.
Then for each $n$, either $U_{\mu_{0}}^{n}$ is empty or $\mu_{0}(U_{\mu_{0}}^{n})=1$.
Therefore, $\bigcap_{n}U_{\mu_{0}}^{n}=\varnothing$.
\end{prop}
\begin{proof}
Since $\mu_{0}\in\probmeas([0,1])$ is $1$-generic and $\mu_{0}\in K$,
there is a basic open ball $B\subseteq\probmeas([0,1])$ such that
$\mu_{0}\in B\subseteq K$. Let $D$ be the set of triples $(\mu,x,s)$
such that $s>0$ and $s\delta_{x}+(1-s)\mu\in B$. Since $D$ is connected,
the proof of Proposition~\ref{prop:unif-seq-tests-on-reals} still
goes through showing that $\mu_{0}(U_{\mu_{0}}^{n})=1$ or $U_{\mu_{0}}^{n}=\varnothing$.
\end{proof}
This would suggest that $\SR{\mu_{0}}{}=[0,1]$ for a 1-generic measure
$\mu_{0}$. However, this next result, also due to Mathieu Hoyrup
(unpublished, personal communication), shows that this is not the
case.
\begin{prop}[Hoyrup]
If $\mu\in\probmeas(\mathbb{X})$ is 1-generic (even weakly 1-generic\footnote{The measure $\mu\in\probmeas(\mathbb{X})$ is \emph{weakly $1$-generic}
if it is in every dense $\Sigma_{1}^{0}$ subset of $\probmeas(\mathbb{X})$.
Equivalently, $\mu$ is weakly $1$-generic if and only if it is in
every dense $\Pi_{2}^{0}$ subset of $\probmeas(\mathbb{X})$ since
a dense $\Pi_{2}^{0}$ set is the intersection of countably many dense
$\Sigma_{1}^{0}$ sets. Clearly, every $1$-generic is weakly $1$-generic. }) and $x\in\mathbb{X}$ is computable then $x\notin\SR{\mu}{}$.
\end{prop}
\begin{proof}
Fix a (weakly) $1$-generic measure $\mu_{0}$ and a computable point
$x_{0}$. Consider the sequence of ``tent functions'' given by $f_{n}(x)=1\dotminus(2^{n}\cdot d_{\mathbb{X}}(x_{0},x))$.
First, we claim that $\int_{{}}f_{n}\,d\mu_{0}<2^{-n}$ for infinitely
many $n$. Indeed, the set $\{\mu:\exists^{\infty}n\ \int_{{}}f_{n}\,d\mu<2^{-n}\}$
is a $\Pi_{2}^{0}$ set. Moreover, the set is dense since there is
a dense set of measures $\mu$ such that $x_{0}\notin\supp\mu$. On
those measures, $\int_{{}}f_{n}\,d\mu=0$ for sufficiently large $n$.
Since $\mu_{0}$ is (weakly) 1-generic we proved our claim.

As for our uniform Schnorr integral test, let 
\[
t_{\mu}(x)=\sum_{n\in\mathbb{N}}\frac{2^{-n}}{\max\{2^{-n},\int_{{}}f_{n}\,d\mu\}}f_{n}(x).
\]
This is lower semicomputable, and the integral $\int_{{}}t_{\mu}\,d\mu$
is computable since 
\[
\int_{{}}\left(\sum_{n>k}\frac{2^{-n}}{\max\{2^{-n},\int_{{}}f_{n}\,d\mu\}}f_{n}(x)\right)d\mu=\sum_{n>k}\frac{2^{-n}\int f_{n}\,d\mu}{\max\{2^{-n},\int_{{}}f_{n}\,d\mu\}}\leq\sum_{n>k}2^{-n}=2^{-k}.
\]
Hence it is a uniform Schnorr integral test. Also, by our earlier
claim, 
\[
t_{\mu_{0}}(x_{0})=\sum_{n\in\mathbb{N}}\frac{2^{-n}}{\max\{2^{-n},\int_{{}}f_{n}\,d\mu\}}=\infty.
\]
Therefore, $x_{0}\notin\SR{\mu_{0}}{}$.
\end{proof}
\begin{rem}
It is not enough to just show that $\mu(\{x\})=0$. See Proposition~\ref{prop:0-can-be-SR-non-atom}.
\end{rem}
\begin{rem}
Despite these negative results, there still are a number of situations
where uniform Schnorr sequential tests are useful. We already saw
in Proposition~\ref{prop:SR-implies-KR} and Theorem~\ref{subsec:seq-test-rel-to-name}
that, when combined with Cauchy names, uniform Schnorr sequential
tests are just as good as uniform Schnorr integral tests. Moreover,
if $K\subseteq\probmeas(\mathbb{X})$ is a sufficiently nice class
of measures, then we can avoid the above counterexamples. For example,
Schnorr \cite[\S 24]{Schnorr:1971rw} gave examples of non-trivial
uniform Schnorr sequential tests $\{U_{\mu}^{n}\}_{\mu\in B}$ where
$B$ is the class of all Bernoulli measures. (Indeed he used this
to define Schnorr randomness for noncomputable Bernoulli measures,
and it seems that his definition agrees with ours.)
\end{rem}

\section{\label{sec:conclusion}Closing comments, future directions, and open
problems}

There are a large number of results concerning Martin-Löf randomness
for noncomputable measures. It is natural to ask which of those results
hold for Schnorr randomness. Here is a brief list of questions and
further topics to explore.
\begin{itemize}
\item To what extent do the results for Schnorr randomness on computable
measures generalize to the noncomputable setting (either noncomputable
measures or oracles)? The results in Section~\ref{sec:useful-characterization}
suggest that many, if not almost all, results will generalize seamlessly.
However, the result at the end of this section suggests that some
caution is needed.
\item What other randomness notions can we relativize is a uniform way to
noncomputable measures? One possibility is computable randomness.
See Rute~\cite[\S 4.3]{Rute:c} for a definition of computable randomness
on noncomputable measures using uniform integral tests. We suspect
that most of the results in this paper also hold for computable randomness
(except for Van Lambalgen's Theorem and its variations, which do not
hold for computable randomness even on computable measures \cite{Bauwens:a}).
\item The results in Sections~\ref{sec:SR-kernels} and \ref{sec:SR-maps}
were only given for continuous kernels, continuous measure-preserving
maps, and continuous conditional probabilities. However, as mentioned
in Remarks~\ref{rem:meas-kernels} and \ref{rem:meas-maps}, one
would like ``continuous'' to be ``measurable.'' There are techniques
in randomness to do exactly this\textemdash including layerwise computability
\cite{Hoyrup.Rojas:2009a,Hoyrup.Rojas:2009b,Miyabe:2013uq}, the computable
metric space of $\mu$-measurable functions \cite[\S 3 on p.36]{Rute:2013pd},
and canonical values \cite[\S 3 on p.36]{Rute:2013pd}. However, it
is not trivial to extend these ideas to a noncomputable measure $\mu$.
\item A probability measure $\mu$ is \emph{neutral for Martin-Löf randomness}
if $\sup_{x\in\mathbb{X}}t_{\mu}(x)<\infty$ for all Martin-Löf integral
tests $t_{\mu}$, and is \emph{weakly neutral for Martin-Löf randomness}
if every point $x\in\mathbb{X}$ is $\mu$-random. Levin (see \cite[\S5]{Gacs:2005})
showed that neutral measures exist when $\mathbb{X}$ is effectively
compact. Moreover, (weakly) neutral measures are noncomputable for
$\mathbb{X}=\{0,1\}^{\mathbb{N}}$ (see \cite{Gacs:2005,Bienvenu.Gacs.Hoyrup.ea:2011,Day.Miller:2013}).
It is clear that every (weakly) neutral measure for Martin-Löf randomness
is (weakly) neutral for Schnorr randomness. Conversely, are there
measures which are (weakly) neutral for Schnorr randomness but not
(weakly) neutral for Martin-Löf randomness? What can be said about
the computability of (weakly) neutral measures for Schnorr randomness?
Day and Miller \cite[Cor.~4.4]{Day.Miller:2013} showed that each
weakly neutral measure $\mu$ has no least Turing degree which computes
$\mu$. We conjecture that a similar result holds for Schnorr randomness,
but possibly replacing computability and Turing degrees with uniform
computability (as in Definition~\ref{def:uniform-comp}) and truth-table
degrees.
\item A point $x$ is \emph{blind} (or \emph{Hippocratic}) \emph{Martin-Löf $\mu$-random}
if $x$ passes all \emph{blind $\mu$-tests}, that is a $\mu$-test
which does not use the measure as an oracle \cite{Kjos-Hanssen:2010aa,Bienvenu.Gacs.Hoyrup.ea:2011}.
For many classes of measures\textemdash e.g.\ the Bernoulli measures\textemdash $\mu$-randomness
and blind $\mu$-randomness are equivalent \cite{Kjos-Hanssen:2010aa,Bienvenu.Gacs.Hoyrup.ea:2011}.
What is the proper definition of a blind Schnorr $\mu$-test? For
example, the Schnorr-Fuchs definition in Subsection~\ref{subsec:Schnorr-Fuchs-def}
is a blind randomness notion. Another approach is to define a blind
$\mu_{0}$-test as a lower semicomputable function $t$ (not dependent
on the measure $\mu_{0}$) such that the map $\mu\mapsto\int_{{}}t\,d\mu$
is computable for all $\mu\in K$ (where $K$ is an effectively closed
set of measures containing $\mu_{0}$). To what extent are these different
notions of blind Schnorr randomness equivalent?
\item If $K$ is an effectively compact class of measures, then $x$ is
\emph{Martin-Löf $K$-random} if $x$ is Martin-Löf $\mu$-random
for some $\mu\in K$ \cite{Bienvenu.Gacs.Hoyrup.ea:2011}. Moreover,
one can define a \emph{Martin-Löf integral $K$-test} as a lower semicomputable
function $t$ such that $\int_{{}}t\,d\mu\leq1$ for all $\mu\in K$.
It turns out that $x$ is Martin-Löf $K$-random if and only if $t(x)<\infty$
for all Martin-Löf integral $K$-tests \cite{Bienvenu.Gacs.Hoyrup.ea:2011}.
What is the appropriate test for Schnorr $K$-randomness? We conjecture
that it is a lower semicomputable function $t$ such that the map
$\mu\mapsto\int_{{}}t\,d\mu$ is computable for $\mu\in K$.
\item Let $K$ be an effectively compact set of probability measures where
$\MLR{\mu}{}\cap\MLR{\nu}{}=\varnothing$ for each pair of measures
$\mu,\nu\in K$. Then given $x\in\MLR{\mu}{}$ for $\mu\in K$, there
is a computable algorithm which uniformly computes $\mu$ from $x$
\cite[\S IV]{Bienvenu.Monin:2012}. (Technically, the algorithm is
``layerwise computable.'') A similar result should also hold of
Schnorr randomness.
\item One of the most important classes of noncomputable probability measures
is the class of Bernoulli measures. Martin-Löf, in his original paper
on randomness \cite[\S IV]{Martin-Lof:1966mz}, considered randomness
on noncomputable Bernoulli measures, using what he called a ``Bernoulli
test.'' His definition (which is quite different from anything in
this paper) is equivalent to the usual definition of Martin-Löf randomness
for noncomputable measures (see \cite[\S IV]{Bienvenu.Gacs.Hoyrup.ea:2011}).
Schnorr \cite[\S 24]{Schnorr:1971rw} adapted Martin-Löf's Bernoulli
tests to Schnorr randomness. As we mentioned in the beginning of Section~\ref{sec:Other-bad-defs},
it seems that Schnorr's definition is equivalent to ours.
\item Suppose that $\mu$ is a computable mixture of probability measures,
that is $\mu(A)=\int_{{}}\nu(A)\,d\xi(\nu)$ where $\xi$ is a computable
probability measure on $\probmeas(\mathbb{X})$. Hoyrup \cite[Thm 3.1]{Hoyrup:2013}
showed that $x\in\MLR{\mu}{}$ if and only if there is a measure $\nu\in\MLR{\xi}{}$
such that $x\in\MLR{\nu}{}$. Such results are important for decompositions
in analysis, such as the ergodic decomposition. We doubt Hoyrup's
result holds for Schnorr randomness in general. (See Question~23
in Rute~\cite{Rute:c}.) Nonetheless, this result likely still holds
in settings such as the ergodic decomposition where $\xi$ is supported
on a pairwise disjoint set of measures.
\item Reimann \cite[Thm.~14, Cor.~23]{Reimann:2008vn} characterized the
points which are strongly $s$-random (a weakening of Martin-Löf randomness)
as the points which are random for a certain class of measures. Similar
results can be found in Diamondstone and Kjos-Hanssen \cite[Cor.~3.6]{Diamondstone.Kjos-Hanssen:2012},
Miller and Rute \cite{Miller.Rute:}, and (implicitly in) Allen, Bienvenu,
and Slaman \cite[\S 4.2]{Allen.Bienvenu.Slaman:2014}.
\item A couple of papers \cite{Reimann.Slaman:2015,Barmpalias.Greenberg.Montalban.ea:2012}
have investigated the sequences in $\{0,1\}^{\mathbb{N}}$ which are
not Martin-Löf random for a continuous probability measure (a probability
measure with no atoms). What can be said about the Schnorr random
version?
\item Reimann and Slaman \cite[Thm.~4.4]{Reimann.Slaman:2015} showed that
a sequence $x\in\{0,1\}^{\mathbb{N}}$ is noncomputable if and only
if $x\in\MLR{\mu}{}$ for some probability measure $\mu$ where $x$
is not a $\mu$-atom. We end this paper by showing that this result
does not hold for Schnorr randomness (Proposition~\ref{prop:0-can-be-SR-non-atom}).
Nonetheless, see Question~\ref{que:reducibility-needed} for a possible
result for Schnorr randomness of a similar flavor.
\end{itemize}
Here $0^{\infty}$ denotes the all-zero sequence $00\ldots$.
\begin{prop}
\label{prop:0-can-be-SR-non-atom}There is a noncomputable probability
measure $\mu$ on $\{0,1\}^{\mathbb{N}}$ such that $0^{\infty}\in\SR{\mu}{}$
but $\mu(\{0^{\infty}\})=0$.
\end{prop}
\begin{proof}
Franklin and Stephan \cite[Thm.~6.3]{Franklin.Stephan:2010} showed
that there exists a sequence $a_{0}\in\{0,1\}^{\mathbb{N}}$ which
is noncomputable, uniformly low for Schnorr randomness (that is, $\SR{\lambda}{}=\SR{\lambda}{a_{0}}$),
and truth-table reducible to some sequence $x_{0}\in\SR{\lambda}{}$.
Since $x_{0}\in\SR{\lambda}{}$ and $a_{0}$ is uniformly low for
Schnorr randomness, we have that 
\[
x_{0}\in\SR{\lambda}{a_{0}}.
\]
Since $a_{0}\leq_{tt}x_{0}$ there is some computable function $f\colon\{0,1\}^{\mathbb{N}}\rightarrow\{0,1\}^{\mathbb{N}}$
such that $f(x_{0})=a_{0}$. By randomness conservation (Proposition~\ref{prop:rand-preserve-cont}),
\[
a_{0}\in\SR{\lambda_{f}}{a_{0}}.
\]
However, $a_{0}$ is not an atom of $\lambda_{f}$ since $a_{0}$
is not computable and all atoms of a computable measure are computable.
Now consider the map $x\mapsto x+_{2}a_{0}$, where $+_{2}$ represents
component-wise addition modulo $2$. Let $\nu$ denote the pushforward
of $\lambda_{f}$ along this map. We have $a_{0}\mapsto a_{0}+_{2}a_{0}=0^{\infty}$.
Since $x\mapsto x+_{2}a_{0}$ is an isomorphism from $\lambda_{f}$
onto $\nu$ and $a_{0}$ is not an atom of $\lambda_{f}$, then $0^{\infty}$
is not an atom of $\nu$. However, by randomness conservation (Proposition~\ref{prop:rand-preserve-cont}),
\[
0^{\infty}\in\SR{\nu}{a_{0}}.
\]
Last, let $\mu=(\nu+2\delta_{a_{0}})/3$. Since $\mu$ is uniformly
equicomputable with the atom $a_{0}$ (as in Definition~\ref{def:sch-unif-rel}),
we have
\[
0^{\infty}\in\SR{\mu}{}.
\]
However, $0^{\infty}$ is not a $\nu$-atom, nor is it equal to $a_{0}$,
so $0^{\infty}$ is not a $\mu$-atom.
\end{proof}
For \emph{computable measures} it is true that $0^{\infty}\in\SR{\mu}{}$
iff $0^{\infty}$ is a $\mu$-atom. However, the proof is nonuniform
(it requires knowing the rate that $\mu[0^{\infty}\upharpoonright n]$
converges to $0$) and cannot be applied to noncomputable measures.

\bibliographystyle{alpha}
\bibliography{schnorr_noncomp_meas}

\end{document}